\documentclass[a4paper,11pt]{article}
\usepackage[margin=3.5cm]{geometry}
\usepackage{mathtools}
\usepackage[utf8]{inputenc}
\usepackage[all]{xy}
\usepackage{float}

\usepackage{amsmath}
\usepackage{amssymb}
\usepackage{amsfonts}
\usepackage{amsthm}

\usepackage{enumerate}
\usepackage{epsfig}
\usepackage{epstopdf}
\usepackage{xcolor,fancybox,graphicx}
\usepackage{url}
\usepackage{psfrag}
\usepackage{color}
\usepackage{algorithm, algorithmic}
\usepackage{pstricks}
\usepackage{graphicx}
\usepackage{tikz}
\usetikzlibrary{calc,arrows}
\usepackage{caption}

\usepackage{stmaryrd}

\usepackage{hyperref}
\usepackage[nameinlink]{cleveref}
\usepackage{authblk}

\newcommand{\lmax}{{l_{B}}}
\newcommand{\lmin}{{l_{0}}}

\newcommand{\loss}{\textrm{Loss}}
\newcommand{\pa}[1]{\mathbf{#1}}

\newcommand{\ol}[1]{\overline{#1}}
\newcommand{\ul}[1]{\underline{#1}}

\renewcommand{\leq}{\leqslant}

\renewcommand{\geq}{\geqslant}

\newcommand{\eps}{\varepsilon}
\newcommand{\ph}{\varphi}
\newcommand{\e}{{\rm e}}
\renewcommand{\d}{ {\, d}}

\newcommand{\eqdef}{\overset{\textrm{def.}}{=}}
\newcommand{\one}{{\mathbf{1}}}
\newcommand{\dps}{\displaystyle}
\def \dd {\mathrm{d}}

\newcommand{\mrm}[1]{\mathrm{#1}}

\newcommand{\E}{\mathbb{E}}

\newcommand{\R}{\mathbb{R}}

\renewcommand{\P}{\mathbb{P}}

\newcommand{\calE}{\mathcal{E}}

\newcommand{\Var}{\mathbb{V}\mathrm{ar}}

\newcommand{\Id}{\mathrm{Id}}
\newcommand{\Law}{{\rm Law}}

\newcommand{\abs}[1]{\left | #1\right |}
\newcommand{\set}[1]{\left\{#1\right\}}

\newcommand{\p}[1]{ \left(#1\right) }
\renewcommand{\b}[1]{ \left [#1\right ] }

\newcommand{\norm}[1]{\left\Vert#1\right\Vert}
\newcommand{\bracket}[1]{\left \langle #1\right \rangle}

\newtheorem{Ass}{Assumption}
\crefname{Ass}{Assumption}{Assumptions}

\newtheorem{The}{Theorem}[section]
\newtheorem{Lem}[The]{Lemma}

\newtheorem{Cor}[The]{Corollary}

\newtheorem{Def}[The]{Definition}

\numberwithin{equation}{section}

\newtheorem{Rem}[The]{Remark}

\title{Fluctuations of Rare Event Simulation with Monte Carlo Splitting in the Small Noise Asymptotics}
\date{}
\author[1]{Fr\'ed\'eric C\'erou}
\author[2]{Sofiane Martel}
\author[1]{Mathias Rousset}

\affil[1]{IRMAR and Inria, University of Rennes, France.}
\affil[2]{\'Ecole Nationale Des Ponts et Chauss\'ees, France.}
\begin{document}
\maketitle

\abstract{Diffusion processes with small noise conditioned to reach a target set are considered. The AMS algorithm is a Monte Carlo method that is used to sample such rare events by iteratively simulating clones of the process and selecting trajectories that have reached the highest value of a so-called importance function. In this paper, the large sample size relative variance of the AMS small probability estimator is considered. The main result is a large deviations logarithmic equivalent of the latter in the small noise asymptotics, which is rigorously derived. It is given as a maximisation problem explicit in terms of the quasi-potential cost function associated with the underlying small noise large deviations. Necessary and sufficient geometric conditions ensuring the vanishing of the obtained quantity  ('weak' asymptotic efficiency) are provided. 
Interpretations and practical consequences are discussed.}

\tableofcontents

\section{Introduction}
Let $\pa{X}^\eps = (\pa{X}^\eps_t)_{t\geq 0}$ denotes a diffusion process 
with small noise parameter $\eps > 0$ and initial condition $\pa{X}_0=x_0$. We are interested in this paper with the simulation of rare events of the form $\set{\tau_{B}(\pa{X}^\eps) < \tau_A(\pa{X}^\eps)}$ where $\tau_B(\pa{x})$ generically denotes the first hitting time of a set $B \subset E$ by the trajectory $\pa{x}$. In the present work, the 'target' set $B$ is defined as the level-set 
$
B = \set{\xi \geq \lmax} = \set{x \in E \mid \, \xi(x) \geq \lmax} 
$
of a continuous function $ \xi: E \to \R$, and the {\em reference} set $A$ typically contains the attractors of the deterministic dynamics $(\pa{X}^0_t)_{t\geq 0}$, and, as such is a recurrent set for the process. \medskip


Such problems are of primary interest in different fields within computational physics. Notable recent examples include molecular simulation (\cite{tutorialAMS,teo2016adaptive}), neutron transport (\cite{louvin2017adaptive}), and climate forecast (\cite{RagWouBou18,LesRagBre18}). In the latter references, the Monte Carlo methods chosen to perform the rare event simulation are Importance Splitting (a.k.a. Multilevel Splitting) type methods with $N$ clones ('fixed effort' algorithms), and are identical or minor variants of the algorithms studied in the present paper. Those algorithms can also be interpreted as Sequential Monte Carlo samplers as studied in~\cite{delmoral06b}, whose structure is defined by a Feynman-Kac model, leading to unbiased estimates of the rare event probability, as studied by P.~Del~Moral in~\cite{delmoral04a} for instance. \medskip

The general idea of Importance Splitting, is to simulate $N$ clones 
$$\pa{X}^{(1,j)}, \ldots, \pa{X}^{(N,j)}$$
using the dynamics of $\pa{X}^\eps$ in a sequential way, $j$ denoting the iteration parameter. At each iteration, the considered algorithm discard trajectories far away from the target set $\set{\xi \geq \lmax}$, and then do split (or branch/duplicate/clone) the trajectories heading closer to $\set{\xi \geq \lmax}$. In order to quantify the closeness to the target set those methods critically rely on the specific choice of the {\em importance function} (also called {\em reaction coordinate}) 
$ \xi: E \to \R $
on $\set{\xi < \lmax}$, the rare of event of interest depending only on the target set $B=\set{\xi \geq \lmax}$. As one may know only very little about the typical trajectories reaching $B$, the specific choice of $\xi$ is usually based on intuitions or loose qualitative knowledge. It is now accepted that this choice is the main bottleneck parameter of the efficiency of those Monte Carlo methods. An optimal theoretical choice of $\xi$ is given by the so called {\em committor function} defined by
$$
\xi^\eps_\ast(x) = \eps \log \P_{x}(\tau_{B}(\pa{X^\eps}) < \tau_{A}(\pa{X^\eps}) ),
$$  
together with $\lmax = 0$. The latter choice, which is in most cases unknown and thus practically infeasible, yields an estimator of the rare event probability with an explicit Poisson distribution and with minimal variance (see \textit{e.g.} Section~$2.2.3$ in~\cite{brehier2015analysis}). We will also use the limiting small noise committor function defined by $\xi_\ast(x)  = \lim_{\eps \to 0} \xi^\eps_\ast(x)$, that will be equal (under our technical assumptions) to $\xi_\ast(x) = U(x,B)$ where $U$ is the subsequently defined two-points quasi-potential function. \medskip 

The main variant we will consider in this work is the so-called Adaptive Multilevel Splitting (AMS) algorithm. The latter can be seen as a limit of a somehow simpler variant -- that we will call the Fixed Multilevel Splitting (FMS) method -- which is an instance of Sequential Monte-Carlo (SMC) sampling. AMS and FMS are detailed in Section~\ref{sec:algos} below. In each iteration of the AMS algorithm, the least performing clone, in terms of the maximum denoted $L$ of the importance function $\xi$ along its trajectory, is discarded; it is then replaced by the duplicate of one of the other trajectories, chosen uniformly among the $N-1$ survivors. The duplication is kept identical from the initial condition up until the first hitting time $\tau_L$ of the level set $\set{\xi \geq L}$, and the duplicated clone is then redrawn independently after $\tau_L$ using the dynamics of $\pa{X}^\eps$.\medskip

Introducing the notation
$
\tau_{l}(\pa{x}) \eqdef \tau_{\set{\xi \geq l}}(\pa{x}),
$
the estimator of the rare event probability associated with level $l$,
\[
 p^\eps_l \eqdef \P\b{ \tau_{l}(\pa{X}^\eps) < \tau_A(\pa{X}^\eps)  },
\]
is given by
\begin{equation*}
p^{N}_{l,\mrm{ams}} \eqdef (1-1/N)^{ I^{N}_{l}}
\end{equation*}
where $I^{N}_{l}$ is the random number of iterations required so that all clones have reached the target set $\set{\xi \geq l}$. The estimator $p^{N}_{l,\mrm{ams}}$
(as well as other non-normalized estimators) is unbiased
$
\E\b{p^{N}_{l,\mrm{ams}}} = p^\eps_l
$ (see \cite{cdgr3,brehier2016unbias}).
The empirical distribution of clones at iteration $I^{N}_{l}$ 
\begin{equation*}
\eta^{N,\mrm{path}}_l = \frac1N \sum_{n=1}^N \delta_{\pa{X}^{ \p{ n,I^{N}_{l} }} },
\end{equation*}
consistently estimates the conditional distribution
$$
\eta^{N,\mrm{path}}_l \xrightarrow[N \to +\infty]{\P} \eta^{\eps,\mrm{path}}_l \eqdef \Law( \pa{X}^\eps \mid \tau_l(\pa{X}^\eps) < \tau_{A}(\pa{X}^\eps) ) .
$$
The product estimator $p^{N}_{l,\mrm{ams}} \eta^{N,\mrm{path}}_l $ is also unbiased. The convergence and asymptotic normality of all the latter estimators, when the number of clone $N$ goes to infinity, $\eps >0$ being fixed, was studied in~\cite{cdgr3} with an explicit expression of the asymptotic variance (see Section~\ref{sec:algos}). The latter is minimal when $\xi = \xi^\eps_\ast$ in which case it is given by $-(p^\eps_{l})^2\ln p^\eps_{l}$ so that the relative asymptotic variance is only logarithmic with respect to the rare event probability. \medskip 

A standard quantity assessing the efficiency of such Monte Carlo algorithms is given by the relative variance times the average computational cost, here at a logarithmic scale:
\begin{equation*}
\calE(p^{\eps,N}_{l,\mrm{ams}}) \eqdef \eps \log \p{ \E\b{ \mrm{Cost}^{\eps,N}_{l,\mrm{ams}}} \Var \b{\frac{p^{\eps,N}_{l,\mrm{ams}}}{p^\eps_l}} }.
\end{equation*}
The efficiency\footnote{Unlike in \cite{budhi-dupuis-book}, the efficiency here is renormalized by the probability of the rare event $p^\eps_l$.} $\calE$ is a simple variance-based criteria properly normalized, in order to be invariant by i) averaging over new independent runs of the full algorithm; ii) multiplication of the estimator by a constant. As pointed out in \cite{lecuyer}, a good feature of an algorithm would be to have a bounded relative variance at fixed computational cost, when the true probability goes to $0$: $p^\eps_l \to_{\eps \to 0} 0$. But this is virtually always out of reach in practical applications, and one will only seek, to the very best, a sub-exponential behavior. For this purpose we resort to the criteria $\calE$ above (adapted the small noise setting thanks to the use of a logarithmic scale). \medskip 

This efficiency criterion has to be compared to i) the behavior given by the crude Monte Carlo estimator obtained by direct simulation of $N$ independent trajectories, for which $\calE(p^{\eps,N}_{l,\mrm{crude}}) \sim_{\eps \to 0} - \eps \log p^\eps$, and ii) the {\em best} behavior, usually referred to as \emph{asymptotically efficient behavior}, for which the relative variance is sub-exponential at fixed cost, that is $\lim_{\eps \to 0} \calE(p^{\eps,N}_l) = 0$, which happens for the AMS algorithm if (but not only if, as will be proved in this work) $\xi = \xi_\ast$. \medskip

The present work is dedicated to the study when $\eps \to 0$ of
$\lim_{N \to + \infty} \calE \p{ p^{\eps,N}_{l,\mrm{ams}} },$ where we stress that the limit $N \to + \infty$ is taken first. This choice considerably simplify the still intricate analysis. We mention that it is an open problem to study the small $\eps$ limit of $\calE \p{ p^{\eps,N}_{l,\mrm{ams}} }$ at fixed $N$. The optimal efficiency obtained after taking first the limit $N \to +\infty$ and then $\eps \to 0$, that is $\lim_{\eps \to 0} \lim_{N \to + \infty} \calE(p^{\eps,N}_{l}) = 0$, will be called here \emph{weak asymptotic efficiency}. We stress that both the order $\lim_{\eps \to 0} \lim_{N \to +\infty}$ and $ \lim_{N \to +\infty} \lim_{\eps \to 0}$ are relevant for practical applications. The case studied here $\lim_{\eps \to 0}\lim_{N \to +\infty}$ is well suited to mild cases where the Monte Carlo algorithm is able to sample the neighbourhood of the least unlikely trajectory defined by the rare event. Converse cases are more difficult.  \medskip

In the specific case of the AMS (or FMC) algorithm, a single trajectory of $\pa{X}$ is refreshed at each iteration, so that it is fair to set the computational cost equal to the total number of algorithmic iterations
$\mrm{Cost}^{\eps,N}_{\lmax} =  I^{\eps,N}_{\lmax} $  (hence the appellation {\em fixed effort} algorithm). When the algorithm is convergent (which happens under mild assumptions), one has $\lim_{N \to + \infty}  I^{\eps,N}_{l} / N = -\log p_l^\eps$ in probability, and since $\log \log p^\eps \ll \log p^\eps $, one obtains that the cost is sub-exponential with respect to $\eps$ and can be removed from the definition of efficiency:
\begin{equation} \label{eq:cost_remove}
\lim_{\eps \to 0} \lim_{N \to + \infty} \calE(p^{\eps,N}_{l,\mrm{ams}}) = \lim_{\eps \to 0} \eps \log\b{ \lim_{N \to + \infty} N
\Var(p^{\eps,N}_{l,\mrm{ams}}/ p_l^\eps ) }.
\end{equation}

The main result of this paper is the rigorous evaluation and interpretation of~\eqref{eq:cost_remove} under some mild technical assumptions, most prominently a Freidlin-Wentzell type uniform large deviations principle on $\pa{X}^\eps$ when $\eps \to 0$. The obtained logarithmic equivalent will be briefly summarized using~\eqref{eq:p_th}-\eqref{eq:loss1} below. The result is based on the explicit formula for the variance of $p^{\eps,N}_{l,\mrm{ams}}$ when $N \to +\infty$ recalled in Section~\ref{sec:algos}, which becomes simpler in the large deviations picture. \medskip

Introducing the notation
$
q_l^\eps(x) \eqdef \P_{\pa{X}^\eps_0=x}[\tau_{l}(\pa{X}^\eps) < \tau_A(\pa{X}^\eps) ]
$
for the probability of the rare event associated with a given level $l$ and initial condition $x$, as well as
$
\eta^\eps_l \eqdef \Law \p{ \pa{X}^\eps_{\tau_l(\pa{X}^\eps)} \mid    \tau_l(\pa{X}^\eps) < \tau_A(\pa{X}^\eps)},
$
the associated distribution of the first hitting place of $\set{\xi \geq l}$ conditioned on occurring before the hitting time of $A$. We will first remark that, for any given initial condition $\pa{X}^\eps_0=x_0$ and target level $\lmax \geq \xi(x_0)$:
\begin{equation}\label{eq:p_th}
\lim_{\eps \to 0} \eps \log\b{ \lim_{N \to + \infty} N 
\Var( p^{\eps,N}_{\lmax,\mrm{ams}} ) } = \sup_{l \in [\xi(x_0),\lmax]} \lim_{\eps \to 0} \eps \log \mrm{Var}_{\eta_{l}^\eps} \p{ q^\eps_{\lmax}  {p^\eps_{l}} / {p^\eps_{l}} },
\end{equation}
which exactly states that, on large deviations scales, the relative variance of the AMS estimator of interest is  given by the largest -- obtained for $l$ spanning the interval  $[\xi(x_0),\lmax]$ -- relative variance of an unbiased (theoretical) estimators of $p^\eps_{l}$. This unbiased theoretical estimator is given by
$p_{l_\ast}^\eps \, q_\eps \p{X^\eps_{l\ast}}$ where $X^\eps_{l_\ast} \sim \eta_{l_\ast}^\eps.$  We will show that critical levels $l_\ast$ do exist and belong to the open interval $]\xi(x_0),\lmax[$.
\medskip

The formula~\eqref{eq:p_th} is interesting in order to interpret the relative variance of the AMS algorithm. We will explain in Section~\ref{sec:related_var} that on of the main contribution in the variance formula~\eqref{eq:p_th} is due to those few trajectories who have been 'lucky' when reaching $\set{\xi \geq l_\ast}$ for the first time because they have a relatively large remaining probability $q^\eps_{\lmax}$ to reach the final level set $B$ before the reference set $A$. Since the empirical distribution of clones at the first hitting time of $\set{\xi \geq l_\ast}$ is approximately $\eta_{l_\ast}^\eps$, the 'lucky' clones do mainly contribute to the $\eta_{l_\ast}^\eps(q_\eps^2)$ quantity in the variance term.\medskip

Our main result is more precise and detailed than this preliminary remark, and consists in proving that
\begin{equation}\label{eq:loss1}
 \lim_{\eps \to 0} \eps \log\b{ \lim_{N \to + \infty} N 
 \Var( p^{\eps,N}_{l,\mrm{ams}} / p_l^\eps) } =  \sup_{l' \in [\xi(x_0),l]}\loss(l')
\end{equation}
where the loss function $\loss(l) \geq 0$ is a non-negative upper semi-continuous function that quantifies the failure from weak asymptotic efficiency. It is explicitly given by the formula
\begin{equation}\label{eq:loss}
\loss(l) \eqdef 2 U(x_0,B) - \inf_{\set{\xi = l}} \b{U^{(l)}(x_0, \, . \, ) + 2 U( \, . \,,B) } - U(x_0,\set{\xi = l}).
\end{equation}
In the above, the notation 
$
U(C_1,C_2) = \inf_{x_1 \in C_1,x_2 \in C_2} U(x_1, x_2)
$ is systematically used, and one has introduced the so-called quasi-potential two-points continuous function $U: \R^d \times \R^d \to \R_+$; a cost function satisfying the triangle inequality $U(x,z) \leq U(x,y) + U(y,z)$ for all $x,y,z$ and defined by 
$$
U(x,y) \eqdef \inf_{\substack{T\geq 0 \\ (\pa{x}_0, \pa{x}_T)=(x,y) \\ \pa{x} \notin A}} I_{[0,T]}(\pa{x}),
$$
where $I_{[0,T]}(x)$ is the rate function associated with the large deviation principle satisfied by $( \pa{X}^\eps )_{\eps > 0}$ on each time interval $[0,T]$. In the same way,
$
U^{(l)}(x, y ) \eqdef \inf_{\substack{ (\pa{x}_0,\pa{x}_{\tau_l(\pa{x})})=(x,y) \\ \pa{x} \notin A}} I_{[0,\tau_l(\pa{x})]}(\pa{x}),
$
is a variant where the minimizing set of trajectories is constrained to take values in $\set{\xi \leq l}$.\medskip

In Section~\ref{sec:loss}, the formula~\eqref{eq:loss} will be interpreted and discussed in details, using a decomposition into an 'underestimation' part and an 'overestimation' part (this nomenclature will be explained); each part being, i) defined by an independent minimization problem involving the intermediate level sets of $\xi$, ii) non-negative. As a consequence, both are identically $0$ if and only if weak asymptotic efficiency is achieved. We will also give in Section~\ref{sec:loss} a geometric interpretation of the loss function. In particular, a \emph{sufficient} condition for weak asymptotic efficiency is given by the following condition: for each $l$ the level sets of the importance function $\xi$ are always 'in between' two tangential level sets ('initial' and 'final') of the quasi-potential $U$. The first 'initial' level set is defined from the initial condition  $\set{U(x_0, \, . \, ) = \mrm{constant}}$, the second 'final level set is defined up to the final rare event set $\set{ U(\, . \,, B ) = \mrm{constant}}$. This condition is equivalent to the existence of a state in $\set{\xi =l}$ that simultaneously minimizes $U(x_0, \, . \, )$ and $ U(\, . \,, B )$. By construction, this state does belong to an optimal trajectory as defined by large deviations of reahcing $B$ before $A$. This sufficient condition for weak asymptotic efficiency is depicted in Figure~\ref{fig:init}. We will show in Section~\ref{sec:insights} that a \emph{necessary} condition is given by the weaker demand that only the minimizer of the quasi-potential \emph{from the initial condition} belongs to an optimal path from the intial condition to $B$ (an 'instanton').\medskip

\begin{figure}
\begin{center}
\begin{tikzpicture}
 
\node [
    above right,
    inner sep=0] (image) at (0,0) {\includegraphics[width=9cm]{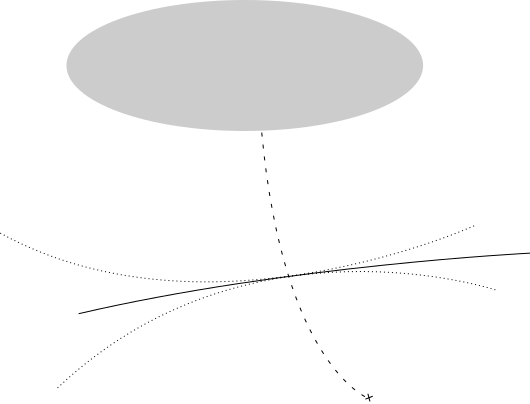}};
 
\begin{scope}[
x={($0.1*(image.south east)$)},
y={($0.1*(image.north west)$)}]
 
%
 
    \node[above right,black,fill=white] at (2.1,7.2){$B=\set{\xi \geq \lmax}$};
    
    
    \draw[latex-,thick,gray] (1.5,3.55) -- ++(0.,+.45)
        node[above,black,fill=white]{ \begin{tabular}{l}
        $\Big \{ U(\, . \,,B) = \mrm{cte}\Big \}$
        \end{tabular}
        };
        
    \draw[latex-,thick,gray] (1.5,0.5) -- ++(0.,-.45)
        node[below,black,fill=white]{$\Big \{ U(x_0,\, . \,)  = \mrm{cte} \Big \}$};
        
    \draw[latex-,thick,gray] (1.8,2.5) -- ++(-0.35,0.)
        node[left,black,fill=white]{$\Big \{ \xi = l \Big \}$};
        
    \node[circle,black,fill=white] at (7.5,0.2){$x_0$};
    
    \draw[latex-,thick,gray] (5.2,4.7) -- ++(+0.35,0.)
        node[right,black,fill=white]{$\set{x_\ast}$};
        
        

 
%
 
\end{scope}
\end{tikzpicture}
\end{center}
\caption{A sufficient condition for weak asymptotic efficiency when satisfied for all $l \in [\xi(x_0),\lmax]$. Note that the level sets of the importance function $\xi$ is 'in between' the level sets of the quasi-potential cost i) from the initial condition, and ii) up to the final set $B$. $\set{x_\ast }$ represents a least unlikely (optimal) trajectory reaching the rare event (the {\em instanton}, see Section~\ref{sec:loss}).}\label{fig:init}
\end{figure}
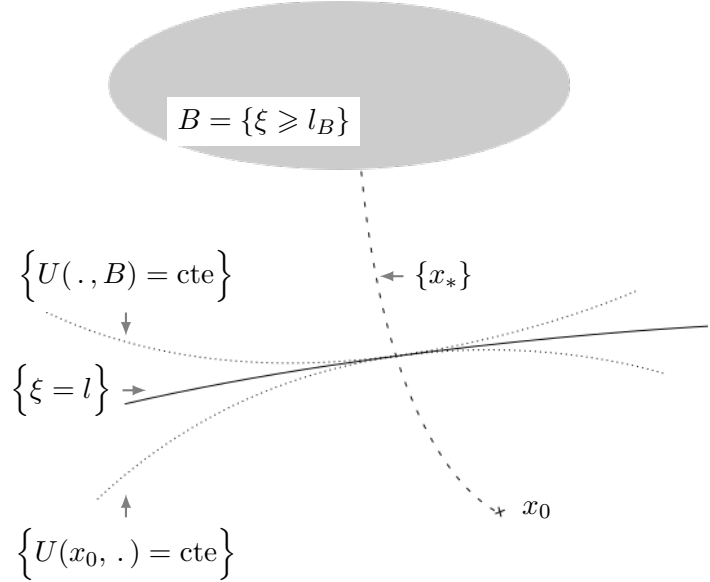

The sufficient condition in Figure~\ref{fig:init} can be related to sub-solutions of the Hamilton-Jacobi equation associated with the quasi-potential cost function $U$ (see Section~\ref{sec:HJ}). It is equivalent to the existence of a strictly increasing real valued function $F$ such that $F \circ \xi$ satisfies a certain weaker notion of sub-solution; this weaker notion of sub-solution is given by the usual global notion, i.e. $ F \circ \xi(y) - F \circ \xi(x) \leq U(x,y)$, $\forall \; x,y$, but restricted to either $x$ given by the initial condition $x=x_0$, or $y$ taking values in the target set $y \in B$. In particular, this is a less demanding condition than being the limiting committor function $\xi = \xi_\ast$, or just $F \circ \xi$ being a sub-solution (which are conditions independent of the initial condition). More comparisons with known results in the literature can be found in Section~\ref{sec:related}. \medskip

Finally, some prospects related to practical applications are discussed in Section~\ref{sec:practice}. 
The general idea is that it may be possible to approximate, at least in the large deviations picture, the loss function $\loss(l)$ in~\eqref{eq:p_th} or the variance formula~\eqref{eq:p_th} using the AMS algorithm or some other \textit{ad hoc} Monte Carlo algorithms. When the latter evaluation can be achieved for various $\xi$ but with a given Monte Carlo procedure constructed with a reference $\xi_0$, one can then try to improve $\xi_0$ by minimizing the obtained quantity over the available choices of $\xi$. \medskip

The paper is organized as follows. In Section~\ref{sec:notations}, the main notations are summarized. Section~\ref{sec:main} provides the mathematical context, the considered assumptions, and the rigorously states our main results. More precisely, Freidlin-Wentzell large deviations theory is recalled in Section~\ref{sec:FW}, while the description and asymptotic normality for large sample size of the considered Monte Carlo algorithms is given in Section~\ref{sec:algos}. Assumptions and results are detailed in respectively Section~\ref{sec:ass} and~\ref{sec:results}. Section~\ref{sec:insights} is dedicated to the interpretation of the results, with comments on few related works. Finally, Section~\ref{sec:proofs} contains the mathematical proofs, in particular the main large deviations analysis and its consequences.

\section{Summary of notations}\label{sec:notations}

$E$ denotes the main Polish state space.\medskip

Level sets are denoted \textit{e.g.} $\set{\xi \leq l} = \set{x \in E \mid \xi(x) \leq l}$, the minimum (resp. maximum) with $\wedge$ (resp. $\vee$).\medskip

Continuous trajectories in state space $E$ are denoted with bold lower case. $\pa{x}$ denotes a generic trajectory, $\pa{X}^\eps$ is a random trajectory with small noise $\eps > 0$ satisfying a LDP in $C([0,T],E)$ for each $T$ with rate function $I_{[0,T]}$.\medskip

We denote entrance times in $S$ with $ \tau_S(\pa{x}) \eqdef \inf \{ t \geq 0 : \pa{x}_t \in S \},$ with the usual convention $\inf\emptyset = + \infty$. When the topology of $S$ is important in order to obtain upper or lower bound, we will use the notation $\tau_S^+ \eqdef \tau_{\mathring{S}}$ and $\tau_S^- \eqdef \tau_{\overline{S}}$. $\xi$ denotes the continuous importance function of interest and we also use the shorthand notation.\medskip

Quasi-potential two-points function $U$ is defined by considering rate function minimizing trajectories avoiding $A$, and with fixed end points, that is
$ U(x,y) \eqdef \inf_{\pa{x},T} I_{[0,T]}(\pa{x})$ under the conditions $\tau_A(\pa{x}) > T$, and $\pa{x}_0=x, \pa{x}_T=y$. In the same way, $U^{(l)}$ denotes the variant where trajectories are also constrained in $\set{\xi \leq l}$, that is,  $U^{(l)}(x,y) \eqdef \inf_{\pa{x},T} I_{[0,T]}(\pa{x})$ under the condition $ T < \tau^+_l(\pa{x}) \wedge \tau_A(\pa{x})$, and $\pa{x}_0=x, \pa{x}_T=y$. If $C$ is a subset of $E$, we will denote
$
U(C,y) := \inf_{x \in C} U(x,y), 
$
$U(y,C) := \inf_{z \in C} U(y,z)$, and similarly for $U^{(l)}$. \medskip
%
%
%
%
%
%
%
%

When $\eps >0$ plays no role, or when considering objects associated with the Monte Carlo algorithm, we may drop in notations the dependence on $\eps$. In the latter case the dependence on $N$ will be used instead.\medskip 

The probability that $\pa{X}$ reaches level $l$ before $A$ for the considered initial condition $x_0$ is denoted $p_l \eqdef \P\b{\tau_{l}(\pa{X}) < \tau_{A}(\pa{X})}$; when we want to stress the dependence in the initial condition we will rather denote 
$q_l(x) \eqdef \P_{\pa{X}_0=x} \left( \tau_l(\pa{X}) < \tau_{A} (\pa{X}) \right)$ so that $q_l(x_0) = p_l$. The conditional distribution at the first hitting time of level $l$ is denoted $\eta_l(\ph) \eqdef \E\b{ \phi\p{\pa{X}_{\tau_l(\pa{X})}}  \mid {\tau_{l}(\pa{X}) < \tau_{A}(\pa{X})} } $, while its pathwise generalization is denoted $ \eta^{\mrm{path}}_l(\psi) \eqdef \E\b{ \psi\p{\pa{X}}  \mid {\tau_{l}(\pa{X}) < \tau_{A}(\pa{X})} } $; $\ph$ (resp. $\psi$) denoting a generic bounded measurable test function on $E$ (resp. path space $C(\R_+,E)$).    \medskip



\section{Framework and results}\label{sec:main}

\subsection{Typical framework: The Freidlin-Wentzell theory}\label{sec:FW}

The classical small noise problem (see Section~$6$ of \cite{Var84}, Section~$5.6$ of~\cite{DemZei98}, \cite{FreiWen12}) is studied for small noise diffusion processes given by the $\R^d$-valued (strong) solution 
$
 \pa{X}^\eps = (\pa{X}^\eps_t)_{t\geq 0}
$
to the Stochastic Differential Equation (SDE) 
\begin{equation}\label{eq:SDE}
\begin{cases}
\dd \pa{X}^\eps_t = b \left( \pa{X}^\eps_t \right) \dd t + \sqrt\eps \sigma \left( \pa{X}^\eps_t \right) \dd W_t, \\
\pa{X}^\eps_0 = x_0
\end{cases}
\end{equation}
where, as usual, $(W_t)_{t\geq0}$ is an $\R^m$-valued Brownian motion; and the drift function $b:\R^d \to \R^d$ as well as the diffusion coefficients $\sigma:\R^d \to \R^{d\times m}$ are Lipschitz continuous. The process is parametrised by the noise amplitude $\eps>0$.\medskip

Let $T > 0$ be a given horizon. It is well-known that the SDE~\eqref{eq:SDE} satisfies a Large Deviations Principle in $C([0,T],\R^d)$ with good rate function (in the case of non-degenerate noise) given by
\[ 
I_{[0,T]}(\pa{x}) = \frac12 \int_0^T \abs{\dot{\pa{x}}_t -b(\pa{x}_t)}_{  g(\pa{x}_t) }^2 \dd t.
\]
The latter rate function is as usual a lower semi-continuous functional on $C([0,T],\R^d)$, finite-valued on a Sobolev sub-space. In the above, $\abs{ \, . \,  }_{g} $ denotes the (Riemannian) $l^2$-norm associated with the metric $$g \eqdef \p{\sigma\sigma^T}^{-1} .$$  Clearly, the rate function has an additive structure in the sense that for any $T' \geq T$, one has
\begin{equation}\label{eq:add}
 I_{[0,T']} = I_{[0,T]} + I_{[T',T]},
\end{equation}
where in the above the two functions of the right hand side have being trivially extended to lower semi-continuous functions on $C([0,T'],\R^d)$. \medskip

In our context, the {\em reference} set $A$ will typically contain the attractor set associated with the limiting ordinary differential equation 
$$ \dot{\pa{x}}_t = b(\pa{x}_t)$$
and the considered initial condition $x_0$ in the sense that $\pa{x}_t \in A$ for all $t$ large enough. \medskip 


Given a reference set $A$, and an initial and final points $x,y$ it is possible to define the \emph{(quasi-potential) cost function} $U:(\R^d)^2\to\R_+$ as the minimum of the rate function:
\[ U(x,y) = U^A(x,y) \eqdef \inf_{T>0} \inf_{\substack{\pa{x}\in C([0,T],\R^d) \setminus \mathring{A}) \\ \pa{x}_0=x\\\pa{x}_T=y}} I_{[0,T]}(\pa{x}) ; \]
the latter quantifies \emph{how unlikely} is a trajectory deviating from the zero-noise solution $\dot x_t = b(x_t)$ in order to go from $x$ to $y$. The quasi-potential 
has a geometric, time-free\footnote{It can be derived by minimizing over an arbitrary time change. A simple calculation enables to double-check that the expression is independent of the path parametrization $\theta \mapsto \pa{x}_{\theta}$.} expression:
\[
 U(x,y) = \inf_{ \substack{ \pa{x}\in C([0,T],\R^d)\\ \pa{x}_0=x,\pa{x}_1=y \\ \pa{x}_{\theta} \notin \mathring{A} }} \int_{0}^1 \left( \abs{ \frac{\d\pa{x}_\theta}{\d \theta}}_{  g(\pa{x}_\theta) } \Big \vert {b(\pa{x}_\theta) \Big \vert }_{  g(\pa{x}_\theta) } - \bracket{ \frac{\d\pa{x}_\theta}{\d \theta}, b(\pa{x}_\theta)}_{  g (\pa{x}_\theta) } \right) \dd \theta ,
\]
and thus can be interpreted as an oriented type of length hence satisfying the triangle inequality
$$
U(x,z) \leq U(x,y) + U(y,z) \qquad \forall x,y,z \in \R^d \setminus \mathring{A} .
$$

Classically (see \cite{freidlin-wentzell-84}), one assumes that the ODE~\eqref{eq:SDE} with $\eps =0$, $\dot x = b(x)$, as a unique attractor $x_A = \lim_{t \to + \infty} x_t$ and the quasi-potential is sometimes defined by the function 
$
U(x) = U(x_A,x).
$
The latter defines the iso-likely exit levels from the attractor defined by $x_A$. 
In that specific context one usually chooses $A$ to be a small neighborhood of $x_A$. In the present work, the {quasi-potential} will rather refer to the two points function $U$. \medskip

Upon some assumptions that will be detailed in Section~\ref{sec:ass}, the quasi-potential provides the asymptotic behavior of the probability of hitting a specific set $B$ before the attractor $A$:
\[
 \lim_{\eps \to 0} \eps \ln \P_{x_0} \b{\tau_B(\pa{X}^\eps) < \tau_A(\pa{X}^\eps) } = - \inf_{y \in B} U(x_0,y).
\]

\subsection{Rare event simulation: The splitting algorithms and their asymptotic variances}\label{sec:algos}

\subsubsection*{Adaptive Multilevel Splitting (AMS)}

The AMS algorithm is described assuming one can simulate the underlying time-continuous diffusion process $(\pa{X}_t)_{t \geq 0}$, but of course in practice the latter has to be discretized and the AMS algorithm needs to be slightly adjusted in consequence.\medskip

The AMS algorithm can be succinctly but rigorously described as follows.\medskip

Initially, $N$ clones (a.k.a. {\em particles}) are simulated independently using the underlying Markov dynamics until they reach the {\em reference} set $A$. They are denoted $(\pa{X}^{(1,0)},\ldots,\pa{X}^{(N,0)})$. \medskip

First i), the iteration index of the algorithm is denoted by the index $i \geq 1$ and is used to enumerate the finite number of levels $L_1 < \ldots < L_i < \ldots L_{I^N_{l}}$ at which a branching event (that is the killing and splitting of well-chosen clones) occurs. $I^N_{l}$ denotes the first iteration at which all clones reach the target set $\set{\xi \geq l}$ for some given $l$. At each iteration $i$, the sample size  \medskip

Second ii), the level $L_{i}$ is computed as the $k$-th order statistics of the {\em scores} associated with each clones; the 'score' being given by the maximum of the importance function $\xi$ over the clones' trajectories, that is: $\sup_{t \leq \tau_A(\pa{X}^{(n,i-1)})} \xi(\pa{X}_t^{(n,i-1)}).$ Then the $k$ clones with lowest scores are killed, and $k$ new clones are uniformly\footnote{\textit{E.g.} with a multinomial or other permutation invariant distribution} randomly picked among the $N-k$ survivors ({\em selection step}). When $k=1$ the AMS algorithm is called the \emph{last particle algorithm}. \medskip
 
Third iii), each newly created clone is modified ({\em mutation step}) using independent simulations of the underlying Markov dynamics, \emph{with initial condition the first hitting time of $\set{\xi \geq L_{j}}$}, and until the reference set $A$ is reached. \medskip

This algorithm yields as an output two main estimators. First,
\begin{equation}\label{eq:estim_p}
p^N_{l,\mathrm{ams}} \eqdef \b{ \frac{N-k}N }^{I^N_{l}}
\end{equation}
estimates (without bias) the rare event probability $\P(\tau_l(\pa{X})<\tau_A(\pa{X}))$; second the empirical distribution of the clones' trajectories at iteration $I^N_l$ denoted
\begin{equation}\label{eq:estim_cond}
\eta^{N,\mrm{path}}_{l,\mrm{ams}} 
\eqdef \frac1N \sum_{n=1}^N \delta_{\pa{X}^{(n,I^N_l)}}
\end{equation}
estimates the conditional distribution $\Law \p{ \pa{X} \mid \tau_l( \pa{X})<\tau_A(\pa{X}) }$. It is also known (see~\cite{cdgr3}), as is also true with any Sequential Monte Carlo algorithm,  that non-normalized quantities are unbiased, that is $\E \b{p^N_{l,\mrm{ams}} \eta^{N,\mrm{path}}_{l,\mrm{ams}}(\Psi) } = \E \b{ \Psi( \pa{X} ) \one_{\tau_1(\pa{X})<\tau_A(\pa{X})} }$ for any pathwise measurable bounded function $\Psi$. \medskip

Under some mild (non-minimal) assumptions (called Assumptions~$1$, $2$ and~$3$ in~\cite{cdgr3}), the above estimators are asymptotically normal.
\begin{The}[Asymptotic normality of AMS,~\cite{cdgr3}]\label{th:CLT} Let $k=1$. Assume that:
\begin{itemize}
 \item $(\pa{X}_t)_{t \geq 0}$ is a Feller Markov process taking values in a locally compact space Polish $E$, with $\xi(\pa{X}_0) \geq \lmin > - \infty$ almost surely.
 \item For all $x \in \set{\xi \in [\lmin,\lmax]}$, entrance times of closures and interiors of sets of interest are the same, that is $\P_{x}\b{\tau_{\xi(x)}^-(\pa{X})=\tau_{\xi(x)}^+(\pa{X})}=\P_{x}\b{\tau_{A}^-(\pa{X})=\tau_{A}^+(\pa{X})}=1$,
 \item $\P_{x}\b{\tau_{\lmax}(\pa{X})<\tau_{A}(\pa{X})}$ is uniformly bounded away from $0$ for $x \in \set{\xi \geq \lmin}$.
\end{itemize}
Then the estimators $ p_{l,\mathrm{ams}}^N$ and $\eta^{N,\mrm{path}}_{l,\mrm{ams}}$ are asymptotically normal when $N \to + \infty$ with $O(1/N)$ variance. Moreover, the asymptotic variance of $\sqrt{N} \, p^N_{l,\mathrm{ams}} $ is given by
\begin{equation}\label{eq:var_in_th}  \sigma_{l,\mrm{ams}}^2 \eqdef -(p_l)^2 \ln p_l + 2 \int_\lmin^l \mrm{Var}_{\eta_{l'}} (q_{l'}) \,p_{l'} \,\dd (-p_{l'}).
\end{equation}
\end{The}
In~\cite{cdgr3}, a similar formula is given for the large sample size variance of all estimators, see Corollary~$2.8$ and Theorem~$2.13$. The extension to the case $k > 1$ under the same assumptions, where $k$ is fixed and $N \to + \infty$ can be obtained using the results of~\cite{cgr_sync}. \medskip

Note that the considered assumptions, although quite mild, are probably not minimal. in particular the Feller assumption may not be necessary and the third assumption may be replaced by the strict positivity $\P_{x_0}\b{\tau_{\lmax}(\pa{X})<\tau_{A}(\pa{X})} > 0$, see Appendix~$F$ in~\cite{cdgr3}. \medskip

Although the AMS algorithm has been originally presented as an \emph{adaptive} Sequential Monte Carlo method, it is more convenient, in order to understand its unbiasedness structure (typical of non-adaptive SMC methods) and to compute formally the variance formula~\eqref{eq:var_in_th}, to recast it as a \emph{classical, non-adaptive, time continuous Sequential Monte Carlo model}. In order to do so the role of {\em time} must be played by the continuum of possible levels in $[\lmin,\lmax]$, see Section~$3$ in~\cite{cdgr3}. The AMS algorithm can then be interpreted as a Fleming-Viot process which possesses the Feynman-Kac structure promoted in the work of P.~Del~Moral (see~\cite{dm00,delmoral04a}). In the next section, we present a fixed levels version of the AMS algorithm which enables to formally justify the latter ideas. A short review of those variants of fixed effort splitting algorithms is provided in~\cite{cgr_hist}. \medskip

\subsubsection*{Fixed Multilevel Splitting} Consider now a fixed number $J$ of levels $\ell_1<\cdots<\ell_J=\lmax$, with final level $\lmax$. Those levels are deterministic and chosen beforehand. The Fixed Multilevel Splitting algorithm is a standard Sequential Monte Carlo method, with a Del~Moral-Feynman-Kac structure as in~\cite{delmoral04a} section 12.2, or see also \cite{alea06}.\medskip 

This algorithm can be succinctly described as follows. Initially, $N$ clones (a.k.a. {\em particles}) are simulated independently using the underlying Markov dynamics until they reach the {\em reference} set $A$. They are denoted $(\pa{X}^{(1,i)},\ldots,\pa{X}^{(N,i)})$ for $i=0$.\medskip

First, i) the iteration index is given by the index $j$ of the considered level in the ladder $\ell_1<\cdots< \ell_j < \cdots <\ell_J=\lmax$. \medskip

Second ii), the {\em selection step} is made using the following $0$~or~$1$ weights: the $K_j$ clones whose {\em score} given the maximum of the importance function $\xi$ over the clones' trajectories -- that is: $\sup_{t \leq \tau_A(\pa{X}^{(n,j-1)})} \xi(\pa{X}_t^{(n;j-1)})$ --  fail to reach level $\ell_j$ are killed (weight $0$). $K_j$ new clones are then  randomly picked (e.g. multinomial distribution\footnote{for a discussion on the different resampling options in that case see \cite{ma2022random}}) among survivors. If $K_j=N$, the algorithm is stopped, and the rare event probability is estimated by~$0$. \medskip

Third iii), each new created clone is modified ({\em mutation step}) by simulating independently the underlying Markov dynamics, \emph{with initial condition the first hitting point of the level set $\set{\xi \geq \ell_j}$}, up until reaching the {\em reference set} A. \medskip

This Fixed Level algorithm yields as an output similar estimators as the Adaptive algorithm; the estimator of the small probability $p_{\ell_j}$ is obtained for instance \textit{mutatis mutandis} by

\[ 
p^N_{\ell_j,\mathrm{fms}} \eqdef \prod_{j'=1}^j \frac{N-K_{j'}}N ;
\]

and the empirical distribution of the clones' trajectories at iteration $j$, that is
$$
\eta^{N,\mrm{path}}_{\ell_j,\mathrm{fms}} \eqdef \frac1N \sum_{n=1}^N \delta_{\pa{X}^{(n,j)}}
$$
estimates, as for AMS, the conditional distribution $\Law \p{ \pa{X} \mid \tau_{l_j}( \pa{X})<\tau_A(\pa{X}) }$. It is also well known and easy to check --as is always true with Sequential Monte Carlo strategies -- that non-normalized quantities are unbiased, that is $\E \b{p^N_{\ell_j,\mathrm{fms}} \eta^{N,\mrm{path}}_{\ell_j,\mathrm{fms}}(\Psi) } = \E \b{ \Psi( \pa{X} ) \one_{\tau_{\ell_j}(\pa{X})<\tau_A(\pa{X})} }$ for any $\Psi$ bounded measurable pathwise test function. \medskip

The AMS algorithm can then be obtained as a \emph{limit of the Fixed Level algorithm} when $J \to + \infty$ with $\max_j \ell_{j+1} - \ell _j  \to 0$; at least in a slightly formal way. To do so, consider in the Fixed Multilevel Splitting algorithm, the random sequence of levels 
$$
L_1 < \ldots < L_i < \ldots L_{I_{\mrm{iter}}}
$$
defined as the subsequence in the sequence $\set{ \ell_1 < \ldots < \ell_J}$ for which at least one killing event occur. The Fixed level algorithm can then be equivalently reformulated by iterating on the index $i = 1 \ldots I_{\lmax}$ instead of $j = 1 \ldots J$ with $\ell_{J}= \lmax$; $I_{\lmax}$ denoting the total number of killing (or branching) events required so that all clones have reached the level $\set{\xi \geq \ell_J= \lmax}$. When $\max_j \ell_{j+1} - \ell_j  \to 0$, then $K_{j'} \in \set{0,1}$ with probability tending to $1$. The Adaptive Multilevel Splitting for $k=1$ (the 'last particle' case) is thus simply obtained by taking the limit $J \to +\infty$ with  $ \max_j \ell_{j+1} - \ell_j  \to 0 $, or even more simply, by formally removing the constraints that $L_i \in \set{\ell_1, \ldots, \ell_J}$. \medskip

The case $k >1$ can be formulated in a similar fashion. One only needs to modify the Fixed Level algorithm above by triggering duplications of clones only when the number of surviving clones gets below $N-k$. The Adaptive Multilevel Splitting for any $k \geq 1$ is again simply obtained by taking the limit $J \to +\infty$ with  $ \max_j \ell_{j+1} - \ell_j  \to 0 $. \medskip

The asymptotic normality of the estimators of the Fixed Multilevel algorithm follows from the classical results in Section~$9$ of~\cite{delmoral04a}.
\begin{The}[Asymptotic normality of Fixed Multilevel Splitting, ~\cite{delmoral04a}] Let $J$ and $\ell_1<\cdots<\ell_J$ be given, assume $\xi(\pa{X}^\eps_0) \geq l_0$ almost surely, and assume that uniformly in the initial condition $x_0 \in \set{\xi \in [\lmin,\ell_J]}$, the probability that $\pa{X}$ reaches $\set{\xi \geq \ell_J}$ before $A$ is bounded away from $0$. \medskip

Then the estimators $ p_{l,\mathrm{fms}}^N$, and $\eta^{N,\mrm{path}}_{l,\mrm{fms}}(f)$, for any test function $f$, are asymptotically normal when $N \to + \infty$ with $O(1/N)$ variance. Moreover, the asymptotic variance of $\sqrt{N} \, p^N_{\ell_J,\mathrm{fms}} $ is given by 
\begin{equation}\label{eq:var_fms}
\sigma^2_{\ell_J,\mrm{fms}} \eqdef \sum_{j=1}^{J-1} \frac{p_{\ell_j}}{p_{\ell_{j-1}}} \left( \left( p_{\ell_{j-1}} \right)^2 - \left( p_{\ell_j} \right)^2 \right) \mathrm{Var}_{\eta_{\ell_j}} (q_{\ell_J}) + (p_{\ell_J})^2 \sum_{j=1}^J \left( \frac{p_{\ell_{j-1}}}{p_{\ell_j}} - 1 \right).
\end{equation}
\end{The}

The above result should hold for $k>1$ fixed, although a rigorous extension is not provided explicitly in the literature up to our knowledge. Note that the variance formula~\eqref{eq:var_fms} is not provided explicitly in Section~$9.4.2$ of~\cite{delmoral04a}, see~\cite{cgr_hist} for more comments. 

It is then possible to derive formula~\eqref{eq:var_in_th} from~\eqref{eq:var_fms} as follows:
\begin{Lem} Assume that the decreasing function $l\mapsto p_l$ is continuous, and that $\ell_J = \lmax$ is fixed. Then one has
\[
 \lim_{\substack{J \to +\infty \\ \max_j \ell_{j+1} - \ell_j  \to 0 }} \sigma^2_{\lmax,\mrm{fms}} = \sigma^2_{\lmax,\mrm{ams}} .
\]
\end{Lem}

\subsection{Assumptions and basic consequences} \label{sec:ass}
Our results can  be stated for a family of pathwise continuous time homogeneous Markov processes $$\left\{ (\pa{X}^\eps_t)_{t\geq 0} : \eps>0 \right\},$$
taking value in a Polish state space $E$. $A$ denotes the {\em reference} set, and $\xi$ the continuous importance function of interest. $\lmax$ denotes the level of the target set $B \eqdef = \set{\xi \geq \lmax}$, the rare event of interest being $\set{\tau_B < \tau_A}$.\medskip

\begin{Rem} Although not necessary, one can assume with a negligible loss of generality that for any initial condition and level $l$, the hitting time of the interior or closure of $A$ or $\set{\xi \leq l}$ are the same $\tau^+_A(\pa{X}^\eps) = \tau^-_A(\pa{X}^\eps)$, and $\tau^+_l(\pa{X}^\eps) = \tau^-_l(\pa{X}^\eps)$, almost surely -- avoiding any ambiguity in the precise definition of the stopping times.
\end{Rem}

We will also need that $\pa{X}^\eps$ satisfies the strong Markov property (for its natural filtration) with respect to the stopping times $\tau^-_l(\pa{X}^\eps)$, $l\in \R$. \medskip 

For simplicity, we assume that the main rare event of interest is defined for a given deterministic initial condition with level greater than a reference $\lmin$
$$
\pa{X}^\eps_0 = x_0, \quad \xi(x_0) > \lmin.
$$
\begin{Rem}
 Our setting and results can be easily generalized \textrm{mutatis mutandis} to a general initial distribution $\eta_0^\eps$ satisfying a Large Deviation Principle on the Polish space $E$. Such a generalization can be obtained by adding the rate function associated to the initial distribution to the rate function $I_{[0,T]}$ of the process on the time interval $[0,T]$ in the small noise asymptotics.
\end{Rem}

The process $\pa{X}^\eps$ is assumed to satisfy on each time interval $[0,T]$ as $\eps\to 0$ a large deviation principle on the Polish space $C([0,T],E)$ (endowed with uniform convergence) for some good rate function $I_{[0,T]}$. The large deviations estimates are (classically) assumed to be true uniformly (in a local sense) with respect to the initial condition. This is the content of our first assumption.

\begin{Ass}\label{ass:ldp}
 For each final time $T$ and initial condition $x_0 \in E$, the family of processes $\set{\pa{X}^\eps}_{\eps \geq 0}$ satisfies a LDP in $C([0,T],E)$ with good rate function $I_{[0,T]}$. The LDP is locally uniform with respect to the initial condition, that is, for the upper bound:
 $$
 \limsup_{\substack{(x,\eps) \to (x_0,0)}} \eps \ln \P_{\pa{X}^\eps_0=x}\b{\pa{X}^\eps \in C} \leq -\inf_{C} I_{[0,T]}
 $$
for any closed set $C \subset C([0,T],E)$; and similarly for the lower bound:
 $$
 \liminf_{\substack{(x,\eps) \to (x _0,0)}} \eps \ln \P_{\pa{X}^\eps_0=x}\b{\pa{X}^\eps \in O} \geq -\inf_{O} I_{[0,T]}
 $$
 for any open set $O \subset C([0,T],E)$. 
\end{Ass}
Classically, the uniform large deviations principle ensures that the Markovian property of the underlying process translates into the additivity property~\eqref{eq:add} of the rate function (using for instance the extended Varadhan lemmas detailed in Section~\ref{sec:varh}):
\begin{Lem} Under Assumption~\ref{ass:ldp}, the family of rate functions $(I_{[T,T']})_{0\leq T \leq T'}$ parametrized by time intervals satisfies the additivity property~\eqref{eq:add}.
\end{Lem}


We need now to define the quasi-potential avoiding $A$.
\begin{Def} Let us denote for each $x,y \in E$ and $A\subset E$
$$
U(x,y) = U^A(x,y) \eqdef  \inf_{\substack{T > 0, \pa{x} \in C([0,T], E \setminus A): \\ \, (\pa{x}_0,\pa{x}_T)=(x,y)}} I_{[0,T]}\b{\pa{x}},
$$
the \emph{quasi-potential} or \emph{likelihood-cost} function to go from $x$ to $y$ while avoiding $A$.
\end{Def}
Next, the following assumption is a very mild technical simplification that prevents degenerate cases in which the boundary of $A$ may play a role in the definition of optimal trajectories.
\begin{Ass}\label{ass:boundA} For any $\delta > 0$, $T > 0$, and any $ \pa{x} \in C([0,T], E \setminus \mathring{A})$ with $\pa{x}_0,\pa{x}_T \notin \partial A)$, there exist $\pa{x}^\delta \in C([0,T], E \setminus \overline{A})$ with $\pa{x}_0=\pa{x}_{0}^\delta$ and $\pa{x}_T=\pa{x}_{T}^\delta$ such that 
$I_{[0,T]}[\pa{x}^\delta] \leq I_{[0,T]}[\pa{x}] + \delta$.
\end{Ass}
In short, Assumption~\ref{ass:boundA} ensures that trajectories avoiding $\mathring{A}$ can be modified to avoid $\overline{A}$ at arbitrarily small cost. The role of this assumption is to make the definition of optimal costs using $\overline{A}$ or $\mathring{A}$ equivalent. Indeed, one immediately gets:
\begin{Rem}
 Under Assumption~\ref{ass:boundA}, one has
$$
U^{\overline{A}} = U^{\mathring{A}}.
$$
\end{Rem}
The present work will also resort to variants of the quasi-potential in which trajectories are restricted to lower level-sets $\set{\xi\leq l}$.
\begin{Def} Let $\xi: E \to \R$ be continuous, and denote for each $l \in \R$, $x \in \set{\xi \leq l}$, $x,y \in \set{\xi =l}$ and $A\subset E$
$$
U^{(l)}(x,y) \eqdef U^{(l),A}(x,y) \eqdef  \inf_{\substack{T > 0, \pa{x} \in C([0,T], \set{\xi \leq l} \setminus A): \\ \, (\pa{x}_0,\pa{x}_T)=(x,y)}} I_{[0,T]}\b{\pa{x}},
$$
the \emph{quasi-potential} or \emph{likelihood-cost} function to go from $x$ to $y$ in $\set{\xi \geq l}$ while avoiding $A$. Note that:
$$
U^{(l)} \geq U .
$$
\end{Def}

\begin{Rem}
Under Assumption~\ref{ass:boundA}, one has
$$
U^{(l),\overline{A}} = U^{(l),\mathring{A}}.
$$
\end{Rem}
It is also useful to remark that costs to reach level sets defined by $U$ or $U^{(l)}$ are the same.
\begin{Lem}
 For each level $l$ and initial condition $X_0$, one has by definition and additivity of the rate function
 $$
 U(x_0,\set{\xi \leq l}) = U^{(l)}(x_0,\set{\xi \leq l})
 $$
\end{Lem}
\begin{proof}
 Minimizers of $U(\,\,.,\set{\xi \leq l})$ can be stopped at the first hitting time of $\set{xi=l}$ to obtain minimizers of $U^{(l)}(\,\,.,\set{\xi \leq l})$.
\end{proof}

We then need a finiteness and continuity assumption on the cost to reach level sets; broadly speaking, ensures that the cost to reach level sets is finite and the cost to 'infinitesimally increase' levels is zero.

\begin{Ass}\label{ass:cont} 
For any intial condition $x\in\set{\lmin \leq \xi \leq \lmax}$, the cost $U(x,\set{\xi \leq \lmax}) < +\infty$ is finite and the cost to immediately enter the open set $\set{\xi > \xi(x)}$ is zero. Formally:
\[
 \inf_{\substack{T > 0, \, \pa{x}:\, \pa{x}_0 = x , \\ \tau^+_{\xi(x)}(\pa{x}) = 0}} I_{[0,T]}[\pa{x}] = 0.
\]
\end{Ass}
In particular, this assumption is the most important assumption required to obtain the continuity (with respect to level) of the cost to enter a level-set; as is stated in Lemma~\ref{lem:cont:UL} below.\medskip

The most demanding assumption is the following. It is similar to Condition (16.22) in \cite{budhi-dupuis-book}. It implies in particular that the reference set $A$ contains all the possible attractors of the dynamics the deterministic dynamics $\pa{X}^{\eps=0}$ with initial condition $x_0$. 
 \begin{Ass}\label{ass:time}
 The process $\pa{X}^\eps$ with initial condition $x_0$ reaches the interior of the reference set $\mathring{A}$ with a probability exponentially close to $1$ when $\eps \to 0$; the associated rate being arbitrary for large enough times. Rigorously:
 $$
 \limsup_{T \to + \infty} \limsup_{\eps \to 0} \eps \ln \P_{\pa{X}^\eps_0=x_0}\b{\tau_{\mathring{A}}(\pa{X}^\eps) > T} = - \infty.
 $$
 \end{Ass}

 \begin{Rem}\label{rem:attract} Assumption~\ref{ass:time} is not satisfied rigorously in many practical situations because practitioners usually do not include in $A$ all the attractors, or even critical points, of the deterministic dynamics $\pa{X}^{\eps=0}$; but only consider those 'close' to the initial condition. However one should remark that:
 \begin{itemize}
  \item When $\eps \to 0$ attractors outside of $A$ will considerably slow down the splitting algorithms, since some trajectories may be stuck in one of the latter for a very large time before reaching $A$. It is a practical argument that shows that Assumption~\ref{ass:time}, although perhaps not minimal, is not a superfluous assumption.
  
  \item It is possible to consider the formal limit $A \to \emptyset$ in the present work, the various minimizations problems of interest -- for instance defining the functions $U$, $U^{(l)}$, or $\loss$, still being well defined (and even sometimes continuous) in this limit.
 \end{itemize}

 \end{Rem}

 The most important consequence of the above assumptions is the following lemma, which interprets the quasi-potential as the rate of vanishing of the rare event probability in the small noise limit.
 \begin{Lem}\label{lem:unif_proba} Let~\Cref{ass:ldp,,ass:cont,,ass:boundA,,ass:time} hold true. Let $x \in \set{\xi \geq \lmin}$, $l\in [\lmin,\lmax]$ be given and define
 $$
 q^\eps_l(x) = \P_{\pa{X}^\eps_{0}=x}[\tau_{l}(\pa{X}^\eps) < \tau_{A}(\pa{X}^\eps)] ,
 $$
 that is the probability to reach the level $l$ before $A$ starting from $x$. Then:
 $$
 \lim_{\substack{\eps \to 0 \\ x \to x_0}} - \eps \log q^\eps_l(x) = U(x_0,\set{\xi = l}) \eqdef  \inf_{\{y : \xi(y)=l\}} U(x_0,y) .
 $$
\end{Lem}
This result is classical. We give a self-contained, warm-up proof adapted to the setting of this work in Section~\ref{sec:unif_proba}. The latter proof justifies the role of the proposed set of assumptions. \medskip

Finally an already mentioned consequence of the above assumptions is the continuity of costs to reach level sets. The proof is also postponed to Section~\ref{sec:unif_proba}.
\begin{Lem} 
\label{lem:cont:UL}Let~\Cref{ass:ldp,,ass:cont,,ass:boundA,,ass:time} hold true. For all $x\in\set{\lmin \leq \xi \leq \lmax}$, the map 
 $
 l \mapsto U\p{x,\set{\xi \geq l}}
 $
is continuous on $[\lmin, \lmax]$.
\end{Lem}

We can now explicitly state a simple example of conditions on $\xi$ and on the finite dimensional SDE~\eqref{eq:SDE} under which the latter assumptions hold true.

\begin{Lem} Consider the SDE~\eqref{eq:SDE} taking values in $\R^d$. Assume $b = - \nabla V$ and that $\sigma = \sqrt{2}\,\Id$.
\begin{enumerate}[1)]
 \item Assume that $\nabla V$ is globally Lipschitz continuous, then Assumption~\ref{ass:ldp} holds true.
 
 \item Assume that $A = \set{\xi \leq l_A}$ with $l_A < \lmin$, $\xi$ is smooth and has no critical point that is
 $
  \nabla \xi \neq 0
 $
 on the set $\set{\xi \in [l_A- \delta, l_A+\delta]}$ for some $\delta$. Then Assumption~\ref{ass:boundA} holds true.
 
 \item Assume that $\xi$ is smooth and has no critical point that is
 $
  \nabla \xi \neq 0
 $
 on the set $\set{\xi \in [\lmin, \lmax]}$, then Assumption~\ref{ass:cont} holds true.
 
 \item Assume that $ \abs{ \nabla V}$ is bounded \emph{away from $0$} on $\R^d\setminus A$. Then Assumption~\ref{ass:time} holds true.
 
\end{enumerate}

%
%
%
\end{Lem}
\begin{proof}
 Item 1) is the classical Freidlin-Wentzell LDP, see~\cite{DemZei98,Var84}. \medskip
 
 Item 2) can be proved by using a diffeomorphism in the neighbourhood of $\partial A$ that approximates the identity. Indeed, since $\xi$ is smooth and has no critical point, $A$ is a smooth domain, and one can locally in the neighbourhood of $\partial A$ consider a smooth set of new coordinates of $\R^d$ the form $(y,r) \in \partial A \times \R$ where $x = y + r \nabla_y \xi$. One can then set $\chi_\delta(x) = y + k_\delta(r) \nabla_y \xi$ where $k_\delta$ is smooth and strictly increasing, $k_\delta(0) > 0$, $k_\delta(r) = r$ outside $[-\delta,+\delta]$ and $\norm{k_\delta'}_\infty < \delta$. Hence $\chi_\delta$ converges to the identity in $C^1$. By construction if $\pa{x}$ does not intersect $\mathring{A}$, then $\chi_\delta(\pa{x})$ does not intersect $\overline{A}$; yet on the other hand $I[\chi_\delta(\pa{x})] \to_{\delta \to 0} I[\pa{x}]$ and the result follows. \medskip
 
 Item 3) is somehow similar to Item 2), yet much simpler. Indeed, just consider the trajectory $\pa{x}_t = t \nabla_{x_0} \xi$ and the result follows. \medskip
 
 Item 4) By assumption $\abs{\nabla V} \geq \kappa > 0$ on set $\R^d\setminus A$. The classical Freidlin-Wentzell formula for the rate function can be rewritten in the form:
 \begin{align*}
  I_{[0,T]}\b{\pa{x}} & = \frac14 \int_0^T \abs{\dot{\pa{x}}}^2 \dd t + \frac14 \int_0^T \abs{\nabla V ( \pa{x}_t) }^2 \dd t + \frac12\p{V(\pa{x}_t) - V(\pa{x}_0)} \\
  & \geq \frac14 \kappa^2 T - \frac12V(\pa{x}_0) \xrightarrow[T \to +\infty]{} + \infty
 \end{align*}

\end{proof}

\begin{Rem} Although it is not done in practice, it might be interesting to include in the reference set $A$ the subset $\set{\abs{\nabla V} \leq \kappa }$ for a well-chosen small $\kappa$, see Remark~\ref{rem:attract}.
\end{Rem}

\subsection{Main results}\label{sec:results}

We can now state rigorously the main results of the present work, whose proof is postponed to Section~\ref{sec:proofs}. The first (and most prominent) result provides the small noise asymptotics of the (large sample size asymptotic) variance formula $\p{\sigma^\eps_{\lmax,\mrm{ams}}}^2$ of the AMS rare event probability estimator $p^{\eps,N}_{\lmax,\mrm{ams}}$, as defined by~\eqref{eq:var_in_th}. We have recalled in Section~\ref{sec:algos} that, under mild assumptions, one has indeed $\p{\sigma^\eps_{\lmax,\mrm{ams}}}^2 = \lim_{N \to + \infty} \frac1N \Var( p^{\eps,N}_{\lmax,\mrm{ams}} )$.

\begin{The}[Small-noise asymptotics of AMS fluctuations] 
\label{th:maintheo}
Let the variance of an AMS probability estimator $\sigma^\eps_{\lmax,\mrm{ams}}$ be defined by~\eqref{eq:var_in_th}. Under~\Cref{ass:ldp,,ass:cont,,ass:time,,ass:boundA} the following holds true:
\begin{align*}
\lim_{\eps \to 0} &\eps \log\b{\p{\sigma^\eps_{\lmax,\mrm{ams}}}^2/p_{\lmax}^2 } 
\\ &= \sup_{l \in [\xi(x_0),\lmax]} \lim_{\eps \to 0} \eps \log \mrm{Var}_{\eta_{l}^\eps} \p{ q^\eps_{\lmax}  {p^\eps_{l}} / {p^\eps_{\lmax}} }
 = \sup_{l \in [\xi(x_0),\lmax]}\loss(l),
\end{align*}
where the loss functions is defined by:
\begin{equation}\label{eq:loss_in_th}
 \loss(l) \eqdef 2 U(x_0,\set{\xi = \lmax}) - \inf_{\set{\xi = l}} \b{U^{(l)}(x_0, \, . \, ) + 2 U( \, . \,,\set{\xi = \lmax}) } - U(x_0,\set{\xi = l}).
\end{equation}
Moreover there exist at least one \emph{critical level} $l_\ast \in ]\xi(x_0),\lmax[$ such that $\sup_{l \in [\xi(x_0),\lmax]}\loss(l)=\loss(l_\ast)$.
\end{The}

The most important result associated with the above theorem is the following sufficient condition for \emph{weak asymptotic efficiency}, which is definition the vanishing of the logarithmic equivalent stated in Theorem~\ref{th:maintheo}.
\begin{The}\label{th:suff} For all $l\in [\xi(x_0),\lmax]$, the loss function~\eqref{eq:loss_in_th} is non-negative: $\loss(l) \geq 0$. If for some $l \in [\xi(x_0),\lmax]$, the initial minimal cost $U(x_0,\set{\xi = l})$ and the final minimal cost $U(\set{\xi=l},\set{\xi = \lmax})$ are attained by a same state $x_\ast(l) \in \set{\xi=l}$, that is $$U(x_0,x_\ast(l)) = U(x_0,\set{\xi = l})$$ and $$U(x_\ast(l),\set{\xi = \lmax})=U(\set{\xi=l},\set{\xi = \lmax}),$$ then the loss function vanishes $\loss(l) = 0$. \end{The}
The above theorem gives a sufficient criterion ensuring that the AMS algorithm is weakly asymptotically efficient, in the sense that $\loss(l) = 0$ for all $l\in [\xi(x_0),\lmax]$. The interpretation and the geometric visualisation of the loss function and of conditions ensuring weak asymptotic efficiency will be discussed in Section~\ref{sec:insights}.In particular, it will be shown that a \emph{necessary} condition for weak asymptotic efficiency is that the minimum of the  
quasi-potential from the initial condition $U(x_0, \, . \,)$ co\"incide with a state of the optimal Freidlin-Wentzell path from $x_0$ to $B$ (the instanton).\medskip
\begin{proof}[Proof of Theorem~\ref{th:suff}] Since $U(x_0,\set{\xi =l})=U^{(l)}(x_0,\set{\xi =l})$, one has on $\set{\xi =l}$:
$$U^{(l)}(x_0, \, . \, ) \leq 2 U^{(l)}(x_0, \, . \, ) - U(x_0,\set{\xi =l}).$$
One can then remark that $$\inf_{\set{\xi = l}} \b{2U^{(l)}(x_0, \, . \, ) + 2 U( \, . \,,\set{\xi = \lmax}) } = 2 U(x_0,\set{\xi=\lmax}), $$
and combining the two last equations we get the positivity of the loss function. \medskip

Now, the existence of $x_\ast(l)$ ensures that for all $x \in \set{\xi = l}$, first i) $U^{(l)}(x_0,x) \geq U(x_0,x_\ast(l))$ and second ii) $U(x,\set{\xi=\lmax}) \geq U(x_\ast(l),\set{\xi=\lmax})$. This shows that the minimization in the definition of the loss functions is attained for $x = x_\ast(l)$, leading to $\loss(l)=0$.
\end{proof}
We also obtain a similar result for the fixed level variant. In what follows, we will use the abuse of notation
$$
U(x,\ell) = U(x,\set{\xi = \ell})
$$
which is clearer in that context to keep track of the different levels.
\begin{The}[Small-noise asymptotics of FMS fluctuations]\label{th:main_fms} Let $\ell_1, \ldots, \ell_J$ denotes a fixed sequence of levels with $\ell_J$. Let the variance of a FMS probability estimator $\p{\sigma^\eps_{\ell_J,\mrm{fms}}}^2$ be defined by~\eqref{eq:var_fms}. Under~\Cref{ass:ldp,,ass:cont,,ass:time,,ass:boundA}, the following holds true:
\[
 \lim_{\eps\to 0} \eps\log \b{ \p{\sigma^{\eps}_{\ell_J,\mathrm{SMC}}}^2/p_{\ell_J,\eps}^2} =  \max(C_1,C_2) \geq 0,
\]
where
\[
 C_1 \eqdef 2 U(x_0,\ell_J)  - \min_{1 \leq j \leq J-1 } \b{ \inf_{\set{\xi=\ell_j}} \p{ U^{(\ell_j)}(x_0,\,.\,) + 2 U(\,.\,,\ell_J) } + U(x_0,\ell_{j-1})},
\]
and
\[
 C_2 \eqdef \max_{1 \leq j \leq J} \b{ U(x_0,\ell_{j}) - U(x_0,\ell_{j-1})} \geq 0 .
\]
\end{The}

It is worth noticing the following:
\begin{Lem}
 The quantity $C_1$ in Theorem~\ref{th:main_fms} satisfies:
 \begin{align*}
 & 0 \leq \min_{1 \leq j \leq J-1} U(x_0,\ell_{j}) - U(x_0,\ell_{j-1}) \\
 &\leq C_1 -\min_{1 \leq j \leq J-1 } \loss(\ell_j) \leq  \max_{1 \leq j \leq J-1} U(x_0,\ell_{j}) - U(x_0,\ell_{j-1}) 
\end{align*}
\end{Lem}
\begin{proof} The lower bound comes from
 \begin{align*}
  &\min_{1 \leq j \leq J-1 } \b{ \inf_{\set{\xi=\ell_j}}\p{ U^{(\ell_j)}(x_0,\,.\,) + 2 U(\,.\,,\ell_J) } + U(x_0,\ell_{j-1}) } +\min_{1 \leq j \leq J-1 } \b{ U(x_0,\ell_{j}) -  U(x_0,\ell_{j-1})} \\
  & \qquad \leq \min_{1 \leq j \leq J-1 } \b{ \inf_{\set{\xi=\ell_j} } \p{ U^{(\ell_j)}(x_0,\,.\,) + 2 U(\,.\,,\ell_J) } + U(x_0,\ell_{j}) } 
 \end{align*}
 while the upper bound comes from
 \begin{align*}
  &\min_{1 \leq j \leq J-1 } \b{ \inf_{\set{\xi=\ell_j}}\p{ U^{(\ell_j)}(x_0,\,.\,) + 2 U(\,.\,,\ell_J) } + U(x_0,\ell_{j-1}) } \\
  & \qquad \geq \min_{1 \leq j \leq J-1 } \b{ \inf_{\set{\xi=\ell_j} } \p{ U^{(\ell_j)}(x_0,\,.\,) + 2 U(\,.\,,\ell_J) } + U(x_0,\ell_{j}) } -  \max_{1 \leq j \leq J-1 } \b{ U(x_0,\ell_{j}) -  U(x_0,\ell_{j-1})} \\
 \end{align*}
\end{proof}
Note that in the above lemma the equality case $C_1 = \min_{1 \leq j \leq J-1 } \loss(\ell_j)$ is satisfied at least if the difference of initial cost between two levels is constant, that is $ U(x_0,\ell_{j}) - U(x_0,\ell_{j-1})$ is independent of $j$ for $j=1 \ldots J$, and thus equal $C_2$. \medskip

The previous lemma thus shows that the surplus of loss in the FMS case as compared to the AMS case exactly comes from the differences $U(x_0,\ell_{j}) - U(x_0,\ell_{j-1})$. The AMS and the FMS small noise asymptotic variance will be similar only if the latter are small as compared to the loss function. This requires to choose sufficiently many levels $\ell_j$ in the FMS algorithms. 



As a corollary we obtain equality between the small noise asymptotic variance of the adaptive AMS and fixed level FMS algorithms, when the number of levels $J$ tends to infinity. This shows that under the different assumptions used in this work (\Cref{ass:ldp,,ass:cont,,ass:time,,ass:boundA} and the assumptions of Theorem~\ref{th:CLT}), one can commute the $J \to + \infty$ and the $\eps \to 0$ limit.
\begin{Cor} Assume $\ell_J = \lmax$ is fixed and $J \to +\infty$ with $\max_j(\ell_j - \ell_{j-1}) \to 0$. Then under \Cref{ass:ldp,,ass:cont,,ass:time,,ass:boundA}
 $$
 \lim_{J \to + \infty} \lim_{\eps\to0}-\eps\log \p{ (\sigma^\eps_{\lmax,\mrm{FMS}} )^2}  = \lim_{\eps\to 0}-\eps\log\p{ (\sigma^\eps_{\lmax,\mrm{AMS}})^2}.
 $$
\end{Cor}


\section{Interpretation and insights}\label{sec:insights}
This section is dedicated to the interpretation of main result of this work, namely the small-noise large-sample-size variance formulas~\eqref{eq:p_th} to \eqref{eq:loss} for the AMS algorithm presented in the introduction, and then stated rigorously in Theorem~\ref{th:maintheo} and Theorem~\ref{th:suff}.\medskip

Throughout this section we will use the notation $B=\set{\xi \geq \lmax}$.

\subsection{The variance formula~\ref{eq:p_th}}\label{sec:related_var}

A result already mentioned in the introduction (formula~\eqref{eq:p_th}) says that the relative variance of the AMS estimator of the probability of interest $p^\eps_{\lmax}$ is equivalent -- at large deviations regime and for the worst possible level $l=l_\ast \in ] \xi(x_0), \lmax [$ -- to the variance of the following unbiased (theoretical) estimator of $p^\eps_{\lmax}$: 
$$\Var \b{ p_{l}^\eps \, q^\eps_{\lmax} \p{X^\eps_{l}}} 
,$$
where $X^\eps_{l} \sim \eta_{l}^\eps$ is distributed according to $\eta_{l}^\eps$, the distribution of the first hitting place of $\set{\xi =l}$ by a trajectory (conditioned to happen before reaching $A$).\medskip

The large deviations estimates obtained in this paper suggests a decomposition of the above variance into the product of $p_{l}^\eps$ on the one hand, and $\gamma_{l}^\eps \p{ (q^\eps_{\lmax})^2 }$ on the other hand -- where $\gamma^\eps_l = p^\eps_l \times \eta^\eps_l$ is the non-normalized version of the conditional distribution $\eta^\eps_l$. This will be discussed in Section~\ref{sec:loss}. We will rather now comment, quite informally, the behavior of the conditional distribution $\eta_{l}^\eps$, of the probability $p^\eps_l$ and their relations to the variance formula above.\medskip

The conditional distribution $\eta_{l}^\eps$ is by definition concentrated in areas of $\set{\xi=l}$ that are the most likely to be reached by trajectories (before $A$). Unfortunately, the importance function $\xi$ usually {\em misleads} (so to speak) trajectories, in the sense that those likely areas of $\set{\xi=l}$ may have a very small remaining probability to hit $B$. In that scenario, we can informally decompose trajectories into two types of events: i) those {\em typical but unuseful trajectories} that hit $\set{\xi=l}$ in the most likely areas but have a very small remaining probability $q^\eps_{\lmax}$ to reach the final level $B$ (before $A$), and ii) rare {\em lucky trajectories} that are outliers with a relatively large remaining probability $q^\eps_{\lmax}$ to reach the final level $B$ (before $A$). \medskip

%
%

By definition, the 'typical but unuseful trajectories' of $\eta_{l}^\eps$ are involved in \emph{underestimation} of the final probability
$$
\eta_l^\eps( q^\eps_{\lmax}  \mid \mrm{typical}) \ll \frac{p^\eps_{\lmax}}{p^\eps_l} ;
$$
while the 'lucky trajectories' ii) are involved in \emph{overestimation}
$$
\eta_l^\eps( q^\eps_{\lmax}  \mid \mrm{lucky}) \gg \frac{p^\eps_{\lmax}}{p^\eps_l} .
$$
Note that in an AMS algorithm, clones sampling 'typical but unuseful trajectories' will have little or no offspring, while clones sampling 'lucky' will have many offsprings and will chiefly contribute to the final estimation. \medskip

In that scenario, the quantity $\mrm{var}_{\eta_{l}^\eps}(q^\eps_{\lmax})$ which quantifies the fluctuations of the function $q^\eps_{\lmax}$ is dominated by 'lucky trajectories' since a large remaining probability will \emph{mainly} contribute to variance through the average square $\eta_{l}^\eps( (q^\eps_{\lmax} )^2)$. \medskip

On the other hand, the total mass of 'typical but unuseful trajectories' is related to $p^\eps_l$: indeed, the larger $p^\eps_l$ is, the easier it is to reach $\set{\xi =l}$, and the more 'typical but unuseful trajectories' will happen. This idea will be made rigorous using the 'underestimation' part of the loss function in Section~\ref{sec:loss} below. \medskip

This informally described phenomenon is somehow similar to what happens with a naive i.i.d. Monte-Carlo. Let $p$ denote the small target probability to be estimated. Assume one estimates $\eps \ll p$ (with probability $1-p_\eps$), or success $1$ (very rarely, with probability $p_\eps$); where $p_\eps = (p-\eps)/(1-\eps) \simeq p$. The overall relative variance is only driven by the rare but highly overestimating value $1$. \medskip 

It should also be noted that in those considerations, the precise value of the underestimation does not influence variance.  To fix ideas in the above simple i.i.d. example, $\eps$ does \emph{not} impact the order of the variance $\simeq p(1-p)$.\medskip

\subsection{A disclaimer about variance}

 As said before, it is well-known that an unbiased estimation of a rare event probability usually leads to a typical systematic \emph{underestimation} (an 'apparent bias'), as well as to a variance driven by rare \emph{overestimations} (where some erroneously large values of the estimator contribute strongly to the variance while being nonetheless rarely seen by the algorithm). This has been noted for instance in \cite{chatterjee2018sample,rolland2015statistical,guyader2020efficient}. 
 This can be seen as a limitations of the present analysis for in some practical cases, if for instance one is only interested in understanding the \emph{ typical underestimation of an AMS algorithm} (as said in the end of the last section this information is lost in our variance analysis) when using the algorithm with limited number of particles. \medskip 
 
 However, our analysis is restricted to a regime where $N$ can be taken to infinity \emph{before} $\eps \to 0$. In that perspective, there are sufficiently many clones to obtain a Central Limit Theorem \cite{cdgr3}, and the variance is a legitimate quantity to discuss the fluctuations of the algorithm. This situation happens in practice when the importance function is sufficiently good to enable a consistent sampling of paths close enough to rate-funtion-optimal trajectories.\medskip 

\subsection{The loss function: discussion}\label{sec:loss}
The loss function can be interpreted using a decomposition into an \emph{underestimation} part and a \emph{overestimation} part, in the spirit of the discussion of Section~\ref{sec:related_var}. The underestimation part, denoted $\loss_{\mrm{U}}$, is associated with the large deviation equivalent of the factor $p^\eps_{l\ast}$ which satisfies
$$
\lim_{\eps \to 0} \eps \ln p^\eps_{l}= - U(x_0,\set{\xi = l}).
$$ 
The overestimation part, denoted $\loss_{\mrm{O}}$, is associated with the large deviation equivalent of the factor $\gamma_{l}^\eps\p{(q^\eps_{\lmax})^2}$ which satisfies
$$
\lim_{\eps \to 0} \eps \ln \gamma_{l}^\eps\p{(q^\eps_{\lmax})^2}
=  - \inf_{\set{\xi = l}} \b{U^{(l)}(x_0, \, . \, ) + 2 U( \, . \,,B) } .
$$ This decomposition is also motivated by the following to facts:
\begin{itemize}
 \item Each part is non-negative and identically $0$ when weak asymptotic efficiency is achieved.
 \item The conditions ensuring that each part is $0$ are mostly independent in terms of the importance function $\xi$, as both are defined by two different minimization problems.
\end{itemize}

In order to be more precise, we can consider for each $l$ a state $x_\ast(l) \in \set{\xi = l} $ such that
$$
U(x_0,B) = U(x_0,x_\ast(l)) + U(x_\ast(l),B). 
$$
A continuous level-indexed path $l \mapsto x_\ast(l)$ satisfying the above condition is called an {\em instanton} in physics literature (\textit{e.g.} \cite{bouchet2014langevin}). 
Then one can consider the decomposition
\[
 \loss(l) = \loss_{\mrm{U}}(l)+\loss_{\mrm{O}}(l),
\]
where we define
\[
 \loss_{\mrm{U}}(l) \eqdef U(x_0,x_\ast(l)) - U(x_0,\set{\xi=l}) \geq 0,
\]
as well as 
\[
 \loss_{\mrm{O}}(l) \eqdef U(x_0,x_\ast(l)) + 2 U(x_\ast(l),B) - \inf_{\set{\xi = l}} \b{U^{(l)}(x_0, \, . \, ) + 2 U( \, . \,,B) } \geq 0.
\]

The underestimation part of the loss function $\loss_{\mrm{U}}(l)$ for a given $l$ is obtained by the minimizing trajectories from $x_0$ to the set $\set{\xi = l}$ which yield the cost $U(x_0,\set{\xi \geq l})$. An example of such a trajectory (assuming it exists for simplicity) is denoted $\pa{x}_{\loss_{\mrm{U}}(l)}$ and is depicted in Figure~\ref{fig:1} and~\ref{fig:2}. A first result is that $\loss_{\mrm{U}}(l) = 0$ if and only if $U(x_0,\set{\xi \geq l}) = U(x_0,x_\ast(l))$, or equivalently, if and only if the level set ${\xi = l}$ is {\em above} the level set $\set{U(x_0, \, . \, ) = U(x_0,x_\ast(l))}$, see Figure~\ref{fig:2}. \medskip

The overestimation part of the loss $\loss_{\mrm{O}}(l)$ for a given $l$ is 
characterized by 
trajectories minimizing the sum of the rate function from $x_0$ to the first hitting time of the level $\set{\xi = l}$ plus \emph{twice} the rate function from the associated entrance point up to $\set{\xi \geq \lmax}$. An example of such a trajectory denoted $\pa{x}_{\loss_{\mrm{O}}(l)}$ is depicted in Figure~\ref{fig:1} and Figure~\ref{fig:2}. Note that $\loss_{\mrm{O}}(l) = 0$ if and only if the composed cost above is attained at the point $x_{\ast}(l)$, see again Figure~\ref{fig:1}. A \emph{sufficient} condition (but not necessary) ensuring $\loss_{\mrm{O}}(l) = 0$ is that the level set ${\xi = l}$ is {\em below} the level set $\set{ U(\, . \,, B ) = U(x_\ast(l),B) } $, see Figure~\ref{fig:2}. This condition is not necessary because contrary to the underestimation part, this overestimation part of the loss involves a competition between an initial cost from $x_0$ and a final cost up to $B$. \medskip

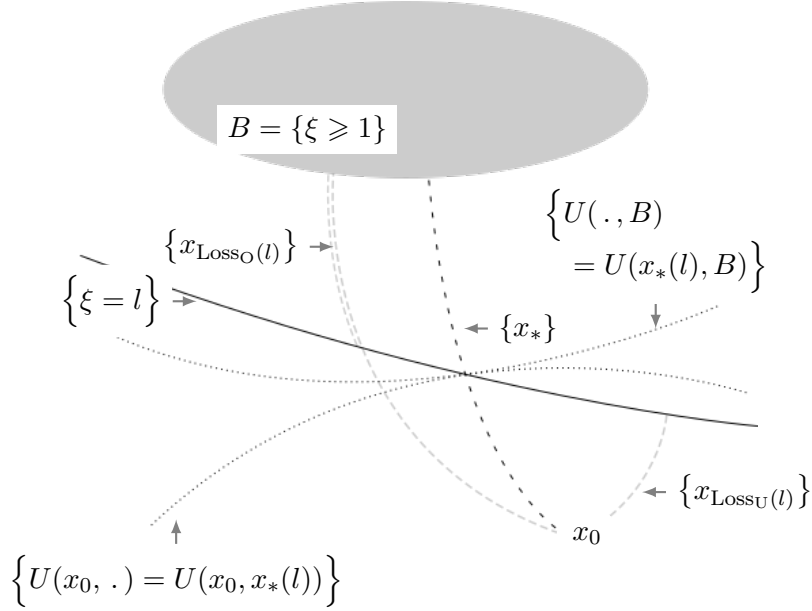
\begin{figure}
\begin{center}
\begin{tikzpicture}
 
\node [
    above right,
    inner sep=0] (image) at (0,0) {\includegraphics[width=9cm]{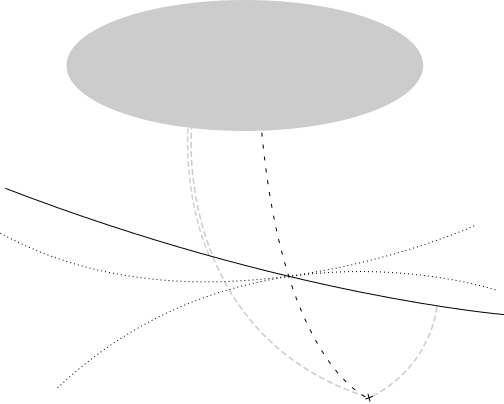}};
 
\begin{scope}[
x={($0.1*(image.south east)$)},
y={($0.1*(image.north west)$)}]
 
%
 
    \node[above right,black,fill=white] at (2.1,7.2){$B=\set{\xi \geq 1}$};
    
    
    \draw[latex-,thick,gray] (8.5,4.0) -- ++(0.,+.45)
        node[above,black,fill=white]{ \begin{tabular}{l}
        $\Big \{ U(\, . \,,B)$ \\ $ \quad = U(x_\ast(l),B)\Big \}$
        \end{tabular}
        };
        
    \draw[latex-,thick,gray] (1.5,0.5) -- ++(0.,-.45)
        node[below,black,fill=white]{$\Big \{ U(x_0,\, . \,) = U(x_0,x_\ast(l)) \Big \}$};
        
    \draw[latex-,thick,gray] (1.8,4.5) -- ++(-0.35,0.)
        node[left,black,fill=white]{$\Big \{ \xi = l \Big \}$};
        
    \draw[latex-,thick,gray] (3.8,5.5) -- ++(-0.35,0.)
        node[left,black,fill=white]{$\set{x_{\loss_{\mrm{O}}(l)}}$};
        
    \node[circle,black,fill=white] at (7.5,0.2){$x_0$};
    
    \draw[latex-,thick,gray] (5.7,4.) -- ++(+0.35,0.)
        node[right,black,fill=white]{$\set{x_\ast}$};
        
    \draw[latex-,thick,gray] (8.3,1.) -- ++(+0.35,0.)
        node[right,black,fill=white]{$\set{x_{\loss_{\mrm{U}}(l)}}$};
        
        

 
%
 
\end{scope}
\end{tikzpicture}
\end{center}
\caption{Graphical interpretation of the main small noise variance formula~\eqref{eq:loss}. The path $\set{x_\ast}$ is the minimizer defining the cost $U(x_0,B)$. The path $\set{x_{\loss_{\mrm{U}}(l)}}$ represents 'typical but unuseful' trajectories, and is the minimizer of the cost $U(x_0,\set{\xi = l})$, that defines the underestimation part of the loss function. The path $\set{x_{\loss_{\mrm{O}}(l)}}$ represents 'lucky' trajectories, and is the minimizer of the cost $\inf_{\set{\xi=l}} U(x_0,\, . \,) + 2 U(\, . \,, B   )$ that defines the overestimation part of the loss function.}\label{fig:1}
\end{figure}

\begin{figure}
\begin{center}
\begin{tikzpicture}
 
\node [
    above right,
    inner sep=0] (image) at (0,0) {\includegraphics[width=9cm]{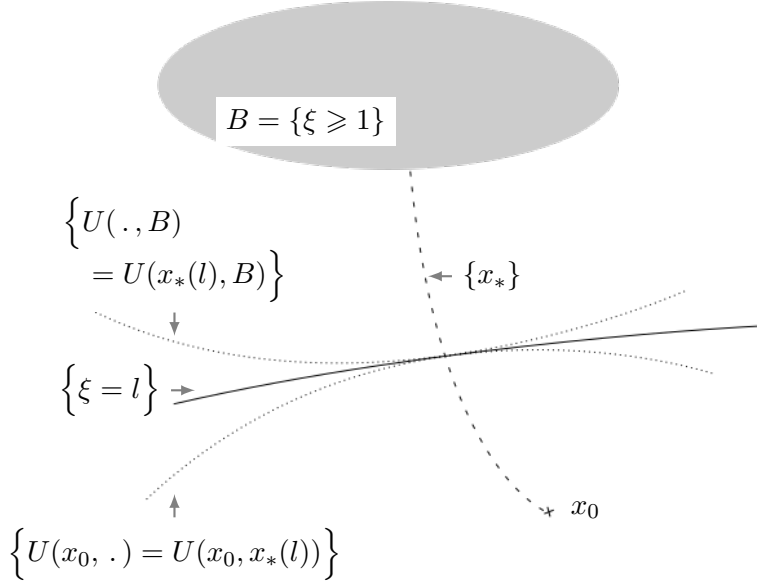}};
 
\begin{scope}[
x={($0.1*(image.south east)$)},
y={($0.1*(image.north west)$)}]
 
%
 
    \node[above right,black,fill=white] at (2.1,7.2){$B=\set{\xi \geq 1}$};
    
    
    \draw[latex-,thick,gray] (1.5,3.55) -- ++(0.,+.45)
        node[above,black,fill=white]{ \begin{tabular}{l}
        $\Big \{ U(\, . \,,B)$ \\ $ \quad = U(x_\ast(l),B)\Big \}$
        \end{tabular}
        };
        
    \draw[latex-,thick,gray] (1.5,0.5) -- ++(0.,-.45)
        node[below,black,fill=white]{$\Big \{ U(x_0,\, . \,) = U(x_0,x_\ast(l)) \Big \}$};
        
    \draw[latex-,thick,gray] (1.8,2.5) -- ++(-0.35,0.)
        node[left,black,fill=white]{$\Big \{ \xi = l \Big \}$};
        
    \node[circle,black,fill=white] at (7.5,0.2){$x_0$};
    
    \draw[latex-,thick,gray] (5.2,4.7) -- ++(+0.35,0.)
        node[right,black,fill=white]{$\set{x_\ast}$};
        
        

 
%
 
\end{scope}
\end{tikzpicture}
\end{center}
\caption{A sufficient condition for weak asymptotic efficiency, given by the simultaneous two conditions $\loss_{\mrm{U}}(l)=\loss_{\mrm{O}}(l)=0$ for all $l \in [\xi(x_0),l_{\mrm{max}]}$. Note that the level sets of the importance function $\xi$ is {\em  between} the level sets of the quasi-potential cost i) from the initial condition, and ii) up to the final set $B$. Mo rover, $\set{x_{\loss_{\mrm{O}}(l)}} = \set{x_{\loss_{\mrm{U}}(l)}} = \set{x_\ast }$ is a (non necessarily unique) instanton.}\label{fig:2}
\end{figure}

Note that the obtained sufficient condition involving $x_\ast(l)$ for weak asymptotic efficiency is exactly the one given in theorem~\ref{th:suff}. Under that condition, minimizer of either $\loss_{\mrm{U}}(l)$ or $\loss_{\mrm{O}}(l)$ can be identified with instantons minimizing the rate function among trajectories reaching $B$ before $A$.  \medskip 

Note also that the geometric conditions ensuring $\loss_{\mrm{O}} = \loss_{\mrm{U}} \equiv 0$ stated above are much weaker than the restriction that $\xi$ is defined by the limiting committor function $\xi = \xi_\ast$.
Note that, contrary to the latter, our conditions depends on the initial condition $x_0$, so that a choice of weakly asymptotic efficient $\xi$ for a given initial condition may not be so for a different initial condition. A related relaxed class of {\em optimal} $\xi$ is given by \emph{sub-solutions of the Hamilton-Jacobi equation} that underlies the rate function of the LDP satisfied by $(\pa{X}^\eps)_{\eps \geq 0}$, see \cite{dean-dupuis-09,dean-dupuis11,cai-dupuis13,budhi-dupuis-book}. This will discussed in a section below. \medskip

The probabilistic and algorithmic interpretation of $\loss_{\mrm{U}}(l)$ for a given $l$ is the following. $\loss_{\mrm{U}}(l)$ will be large when 
when the minimizers associated with the optimal cost $U(x_0,\set{\xi = l})$ do not correspond to $x_{\ast}(l)$ where $x_{\ast}$ is an instanton (a minimizer associated with the global cost $U(x_0,\set{\xi = \lmax})$). This means that the conditional distribution $\eta^\eps_{l}$ which is concentrated towards the minimizers associated with $U(x_0,\set{\xi = l})$ (by a standard Gibb's conditioning argument in large deviations theory). $\loss_{\mrm{U}}(l)$ thus quantifies the likelihood of 'typical but unuseful trajectories', as discussed in Section~\ref{sec:related_var}. The large deviation picture is depicted in Figure~\ref{fig:2}. Those trajectories are eventually associated with underestimation. We stress that the quantity $\loss_{\mrm{U}}$ is not related to the specific value of this underestimation (the latter is rather encoded by the function $U(\,.\,,B)$ evaluated at the minimizer associated with $U(x_0,\set{\xi = l})$, which does not appear in the definition of the loss function). In an AMS algorithm, $\loss_{\mrm{U}}(l)$ can be associated with the proportion of clones that will be quickly killed after having reached the level $l$.  





The algorithmic interpretation of $\loss_{\mrm{O}}(l)$ for a given $l$ is complementary. We have seen that 
it is associated with large values of the quantity $\gamma_l^\eps((q^\eps_{\lmax})^2)$. As discussed in Section~\ref{sec:related_var}, this term can be associated with specific {\em lucky} trajectories which turn out to be very contributive to variance in the end because they can reach the rare event set $\set{\xi \geq \lmax}$ with a relatively large probability $q^\eps_{\lmax}$. This is quantified in the loss function by the variational problem:
$$
 \inf_{\set{\xi = l}} \b{U^{(l)}(x_0, \, . \, ) + 2 U( \, . \,,B) }
$$
in which the final cost $U( \, . \,,B)$ counts \emph{twice} as compared to 
$$
 \inf_{\set{\xi = l}} \b{U^{(l)}(x_0, \, . \, ) + U( \, . \,,B) } = U(x_0,B)
$$
which is minimized by global minimzers $x_\ast$. This implies that states with lower final cost $U( \, . \,,B)$ are much preferred. These states defines the 'lucky outliers' discussed in Section~\ref{sec:related_var}. A typical lucky trajectory is depicted in Fig.\ref{fig:2}. In an AMS algorithm, $\loss_{\mrm{U}}$ is associated with the overestimation by the small fraction of clones that are the most likely to reach $B$ after having reached the level $l$. 

%
%
%
%

\subsection{Summary}

 To summarize our main results, we proved that there exists a \emph{critical level} $l_\ast$ that will contribute mostly to variance. This contribution is described by the distribution $\eta_{l_\ast}^\eps$ of trajectories at the first hitting times of $l_\ast$ and the associated probability $p_{l_\ast}^\eps$ to reach level ${l_\ast}$ (all before $A$). The logarithmic equivalent of the relative variance can also be decomposed into two independent non-negative terms. The first term, $\loss_{\mrm{U}}({l_\ast})$, quantifies the likelihood to have 'typical but unuseful' states in the distribution $\eta_{l_\ast}^\eps$. The second term, $\loss_{\mrm{O}}({l_\ast}) \geq 0 $, quantifies the overestimation by outliers in $\eta_{l_\ast}^\eps$ that are likely to eventually reach $B$. We also provide a simple geometric sufficient (resp. necessary and sufficient) condition on $\xi$ depicted in Figures~\ref{fig:1} and~\ref{fig:2} such that $\loss_{\mrm{O}} \equiv 0$ or $\loss_{\mrm{U}} \equiv 0$. 
 
 \subsection{Interpretation as a Hamilton-Jacobi sub-solution}\label{sec:HJ}


We have thus obtained a simple geometric sufficient criteria for weak asymptotic efficiency ($\sup_l\loss(l) = 0$): for each $l$ the level set ${\xi = l}$ lies in between the iso-cost set from $x_0$, as depicted in Figure~\ref{fig:2}. Formally, this amounts to the the existence of an increasing function $F$ (realized by the cost along $x_\ast$) such that
\begin{equation}\label{eq:suff_bis}
\begin{cases}
F(\xi(x_0)) - F(\xi(y)) \leq U(x_0,y), \, &\forall y \in \set{\xi \in [\xi(x_0),\lmax]} \\
F(\xi(x)) - F(\xi(y)) \leq U(x,y), \, &\forall (x,y) \in \set{\xi \in [\xi(x_0),\lmax]} \times \set{\xi \geq \lmax} 
 \end{cases}
\end{equation}
This expression is related to \emph{sub-solutions of the Hamilton-Jacobi equation} that underlies the rate function of the LDP satisfied by $(\pa{X}^\eps)_{\eps \geq 0}$ (see \textit{e.g.} \cite{dean-dupuis-09,dean-dupuis11,cai-dupuis13,budhi-dupuis-book}). A sub-solution $f$ is a function that satisfies the inequality $f(x)-f(y) \leq U(x,y)$ for all  $x,y \in \set{\xi \leq \lmax}$. This is however a much more demanding condition than~\eqref{eq:suff_bis} because: i) we do not need to compute the reparametrization $F$ which be given by $F(\xi(x_\ast(l))) = U(x_\ast(l),B) + \mrm{cte}$; and ii) the sub-solution inequality need to be true \emph{only} for initial points in the initial condition of $\pa{X}^\eps$, and final points in $B$.  See also Section~\ref{sec:related} for comments on related work.


\subsection{Possible practical consequences}\label{sec:practice}
Finally, there are several consequences of our results for practical purposes. We list them below. All are left for future work. \medskip

The first question is about estimating the variance~\eqref{eq:p_th} or equivalently the loss~\eqref{eq:loss} which are the main result of this paper.

\begin{itemize}
 \item An idea is to try to estimate a non-asymptotic (with respect to $\eps$) form of the variance formula~\eqref{eq:p_th} after one realisation of the AMS algorithm. Variance estimation has recently been studied in \cite{chan2013general,lee2018variance,du2021variance}. The present work suggests that variance estimation can be considerably simplified at the cost of being accurate only asymptotically for large $N$ and small $eps$. For instance, one can first estimate on the states $x$ visited by the clones the probability $q_\eps(x)$ to reach the final set $\set{\xi \geq \lmax}$. This can be done using the genealogy of the clones and the formula~\eqref{eq:estim_cond}. One can then proceeds using the formula~\eqref{eq:p_th} by averaging those estimations over $\eta_l^{\eps}$ for each $l$, and then by minimizing on the level $l$.

 \item We will also remark in Section~\ref{sec:N=2} that the loss function can be expressed as a new, simulable, rare event probability. The latter is the probability that an AMS algorithm with $N=2$ clones succeeds in at most one iteration. This fact may be used to estimate $\sup_{l}\loss(l)$ using a secondary Monte Carlo rare event algorithm, more appropriate than the first one since it will purposely simulate the rare clones involved in the overestimation (that drives the variance) of the final probability.
 \end{itemize}
 
 \begin{itemize}
 \item The variance estimation proposed in the first item above can in fact estimate the variance obtained with various importance function $\xi$, using, say, an AMS algorithm performed with a given reference $\xi_0$. Indeed, the only quantity depending on $\xi$ is the first hitting place $\eta_l^\eps$ associated with set $\set{\xi \geq l}$, but this can be estimated using the full genealogical estimator~\eqref{eq:estim_cond}.
 
 \item There is currently a lot of effort in practical applications aimed at optimizing the importance function $\xi$ in order to obtain reliable results, see for instance the references in the review for molecular simulation applications~\cite{rogal2021reaction}, or the paper~\cite{lucente2022coupling} that uses a data-driven approach. Our analysis provides insights on the minimal conditions an importance function must satisfy in order to provide efficiency. In particular, instead of trying to exactly compute the committor function $\xi_\ast$, one may try, after a rare event simulation, to update $\xi$ by trying to minimize the rough estimation of the variance as discussed in the previous item. 
 
\end{itemize}

\subsection{Comparison to previous work}\label{sec:related}

The idea to analyse rare event (multi-level) splitting Monte Carlo simulation algorithms in a large deviation setting has been mainly developed by P. Dupuis and his co-authors. 

In~\cite{dean-dupuis-09} a fixed multilevel splitting method with varying number of clones is studied. The main difference from our study is that the clones do not interact with each other through the splitting mechanism and the rate of splitting is given by the variations of the values of the level function. The authors then show that a sufficient condition to obtain asymptotic efficiency in a large deviation small noise limit is that the level function must be a sub-solution of the Hamilton-Jacobi problem associated with the Lagrangian formulation of the rate function (see Section ...). They also argue that this condition should be necessary. This type of condition on the level function (or importance function) is reminiscent to the type of condition required on the importance function in importance sampling in order to achieve asymptotic efficiency (see \textit{e.g.}~\cite{dupuis-wang07,dupuis-wang09,dupuis-wang04}); note that the importance function must in addition be a smooth sub-solution which highlights the likely generic lack of robustness of importance sampling as compared to importance splitting. \medskip

As compared to our work, those result are more precise in the sens that the number of clones (although randomly varying) is finite, whereas our analysis is restricted to asymptotic (in terms of clones sample size) variance. However, our work suggest three (related with each other) improvements enabled by the adaptivity of levels and the fixed number of clones:
\begin{itemize}
 \item Weak asymptotic efficiency can be achieved if \emph{some parametrization} of the level function satisfies a certain weaker condition related to sub-solutions of the considered Hamilton-Jacobi problem. This condition is sufficient but not necessary.
 \item The notion of sub-solution is weaker: it only has to be one with respect to the support of the initial condition and the final target set (and not for every pair of points in space).
 \item We do not face the problem of explosion or implosion of the total number of clones. In~\cite{dean-dupuis-09} the splitting rate has to be tuned carefully -- close to the inverse of probability of transitions between levels -- to avoid such population size issues, even if the population size can still grow polynomially. 
\end{itemize}

Similarly~\cite{dean-dupuis11} studies a variant called RESTART which enables to reduce the trajectory length of most of the clones. The analysis is also improved (the notion of sub-solution is defined variationally instead of as a viscosity solution of a PDE) and then recapitulated in~\cite{budhi-dupuis-book}. In~\cite{cai-dupuis13}, the authors nonetheless studies a splitting algorithm with fixed number of clones, but in dimension one only (asymptotic efficiency is then conditionless).

\section{Large Deviations estimates and proof of the main result}\label{sec:proofs}
This Section is devoted to the the large deviations estimates that eventually lead to the main result theorem~\ref{th:maintheo}.
\subsection{Stopping times and topology}
We start that a key technical remark on the semi-continuity of stopping times.
\begin{Lem}\label{lem:lsc} Let $E$ be a Polish space, $T >0$ be given, and $A,B \subset E$. The map $\pa{x} \mapsto \tau_A^-(\pa{x})$ (resp. $\tau^+_A$) is lower (resp. upper) semi-continuous as a function of $C([0,T],E)$ with the topology of uniform convergence. In particular,
 $$
 \set{\pa{x} \mid \tau_A^+(\pa{x}) < \tau_B^-(\pa{x})}
 $$
 is an open subset of $C([0,T],E)$.
\end{Lem}
\begin{proof}
 See Lemma~$A.4$ in~\cite{cdgr2}. 
\end{proof}

\subsection{Continuity of the cost to reach level sets}
Let us prove Lemma~\ref{lem:cont:UL}. \medskip.

\underline{Right continuity.} Let $l_\ast$ be a given level. By definition of $U$, there is for each $\delta > 0$ a trajectory $\pa{x}^\delta$ with $U(x_0,\set{\xi \geq l_\ast}) + \delta \geq I_{[0,\tau^-_{l_\ast}(\pa{x}^\delta)]}[\pa{x}^\delta]$. By Assumption~\ref{ass:cont}, it is possible to strictly extend $\pa{x^\delta}$ so that $\tau^-_{l_\ast}(\pa{x}^\delta)=\tau^+_{l_\ast}(\pa{x}^\delta)$ and $U(x_0,\set{\xi \geq l_\ast}) + 2\delta \geq I_{[0,T]}[\pa{x}^\delta]$. This implies that there exists a $l = \max_t \xi(\pa{x}^\delta_t) > l_\ast$ with $U(x_0,\set{\xi \geq l_\ast}) + 2\delta \geq U(x_0,\set{\xi \geq l})$. The result follows since $\delta$ is arbitrary and the function non-decreasing.\medskip

\underline{Left continuity.}  Let $l$ be an arbitrary level,  $l_k$ be increasing with $k$ and converging to $l$, and $\delta > 0$ be given, arbitrary. By definition of $U$ and by Assumption~\ref{ass:time}, there is a time horizon $T>0$ and a sequence of paths $\pa{x}^k$ such that $I_{[0,T]}[\pa{x}^k] \leq U(x_0,\set{\xi \leq l_k}) + \delta$. Since $I$ is a good rate function in Assumption~\ref{ass:ldp} (lower semi-continuous with compact pull-back of closed bounded above intervals), one can extract a (uniformly) converging sub-sequence of paths such that $I_{[0,T]}[\pa{x}^\infty] \leq \liminf_k I_{[0,T]}[\pa{x}^k]$. By construction $\liminf_k I_{[0,T]}[\pa{x}^k] \leq \lim_k U(x_0,\set{\xi \leq l_k}) + \delta $, and by continuity of $\xi$, $\tau^-_l(\pa{x}^\infty \leq T)$ which implies $I_{[0,T]}[\pa{x}^\infty] \geq U(x_0,\set{\xi \geq l})$, hence the result.

\subsection{Small noise asymptotics of the rare event}\label{sec:unif_proba}

We can now turn to the proof of Lemma~\ref{lem:unif_proba}. We recall that the latter states that under~\Cref{ass:ldp,,ass:cont,,ass:boundA,,ass:time}, then
$
 \lim_{{\eps \to 0 ,\, x \to x_0}} \eps \ln \P_{x}\b{\tau_l(\pa{X}^\eps) < \tau_A(\pa{X}^\eps)} = - U(x_0,\set{\xi\geq l}) .
$

\begin{proof}[Proof of Lemma~\ref{lem:unif_proba}] {\,} \medskip

 \textbf{Upper bound.} Assume $T > t_\ast$ given, arbitrary large. We can then consider the upper bound
\begin{align*}
 \P_{x}\b{\tau_l(\pa{X}^\eps) < \tau_A(\pa{X}^\eps)} 
 & \leq \P_{x}\b{\tau_l^-(\pa{X}^{\eps}) \leq \tau_A^+(\pa{X}^{\eps}) \wedge T } + \P\b{\tau_A^+(\pa{X}^{\eps}) \geq T } 
\end{align*}
 On the one hand, using Assumption~\ref{ass:time}, one has $\limsup_{\eps} \eps \ln \P\b{\tau_A^+(\pa{X}^{\eps}) \geq T }  \leq - C_T$, with $\lim_{T+\infty} C_T = +\infty$. 
 On the other hand Lemma~\ref{lem:lsc} ensures that the set $\set{\pa{x}: \tau_l^-(\pa{x}^{}) \leq \tau_A^+(\pa{X}^{\eps}) \wedge T}$ is closed in $C([0,T],E)$. 
 Thus, the uniform LDP upper bound yields
 \begin{align*}
 -\limsup_{\substack{\eps \to 0 \\ x \to x_0}} \eps \ln \P_{x}\b{\tau_l(\pa{X}^\eps) < \tau_A(\pa{X}^\eps)} &\geq \min \b{ \inf_{\substack{\pa{x}\in C([0,T],E) \\ \pa{x}(0)=x_0\\\tau_l^-(\pa{x}) \leq T}} I_{[0,T]}\b{\pa{x} } , C_T } \\
 &\geq \min \p{ U(x_0,\set{\xi=l}),C_T},
\end{align*}
 $C_T$ being arbitrary large.
 
 \paragraph{Lower bound} 
Let $l\geq 0$ be a given level, $\delta > 0$ be given, arbitrarily small, and let $x_0 \in \set{\xi \geq \lmin}$ be an initial condition. Using the continuity property of Lemma~\ref{lem:cont:UL}, there is a small enough $h > 0$ such that $U(x_0,\set{\xi=l+h}) \leq U(x_0,\set{\xi=l}) + \delta/2$. One can thus construct a minimizing continuous path $\pa{x}_\ast$ with $\pa{x}_\ast(0)=x_0$ and a time $t_\ast \geq 0$ such that
\begin{itemize}
 \item[i)] $I_{[0,t_\ast]}\b{\pa{x}_\ast} \in [U(x_0,\set{\xi=l}),U(x_0,\set{\xi=l})+\delta/2]$,
 \item[ii)] $\xi(\pa{x}_\ast(t_\ast)) > l$ so that $\tau_l^+(\pa{x}_\ast) < t_\ast$,
 \item[iii)] $\tau^-_A(\pa{x}_\ast) > t_\ast$.
\end{itemize}
Extending the trajectory $\pa{x}_\ast$ with a $\delta/2$ minimizer of the rate function $I_{[t_\ast,T]}\b{\pa{x}_{t_\ast}}$, we obtain a path which satisfies for any arbitrary large final time $T > t_\ast$, $I_{[0,T]}[\pa{x}_\ast] \in [U(x_0,\set{\xi=l}),U(x_0,\set{\xi=l})+\delta]$. \medskip


We can then consider the lower bound
\begin{align*}
 \P_{x}\b{\tau_l(\pa{X}^\eps) < \tau_A(\pa{X}^\eps)} 
 \geq \P_{x}\b{\tau_l^+(\pa{X}^{\eps}) < t_\ast  < \tau_A^-(\pa{X}^{\eps})\wedge T } .
\end{align*}
Applying the uniform LDP to the open (see Lemma~\ref{lem:lsc}) set $$\set{\pa{x} \in C([0,T],E)\mid \tau_l^+(\pa{x}^{}) < t_\ast  < \tau_A^-(\pa{x}^{}) } $$ it yields
\begin{align*}
 - \liminf_{\substack{\eps \to 0 \\ x \to x_0}} \eps \ln \P_{x}\b{\tau_l(\pa{X}^\eps) < \tau_A(\pa{X}^\eps)} &\leq \inf_{\substack{\pa{x}\in C([0,T],E) \\\pa{x}(0)=x_0\\  \tau_l^+(\pa{x}) < t_\ast < \tau_A^-(\pa{x})}} I_{[0,T]}\b{\pa{x} } \\
 &\leq I_{[0,T]}\b{\pa{x}_\ast } \leq U(x_0,\set{\xi=l})+\delta,
\end{align*}
$\delta$ being arbitrary small.

\end{proof}

\subsection{Varadhan lemmas}\label{sec:varh}
In this section, we state and prove minor variants of the classical Varadhan lemmas, in the case where the potential function is general ($\eps$-dependent and only measurable). \medskip

We start with the easier lower bound.
\begin{Lem}
\label{lem:var:lower}
 Assume $\p{\pa{X}^\eps}_{\eps > 0}$ satisfies a LDP with rate function $I$ on a Polish state space. Let $\p{V_\eps}_{\eps > 0}$ be a family of measurable functions in $]-\infty,+\infty]$. Define for each state $x$ the upper semi-continuous envelope $V_+(x)$ of the latter by:
 \begin{equation}\label{eq:up_V}
 V_+(x) \eqdef \limsup_{\substack{\eps \to 0 \\ y \to x}} V_\eps(y) .
 \end{equation}
 Then
 $$
 -\liminf_{\eps \to 0} \eps \ln \E\b{ \e^{- \frac1\eps V_\eps}(\pa{X}^\eps)} \leq \inf\p{ I + V_+ } .
 $$
\end{Lem}
\begin{proof}
Similar to the usual proof of Varadhan's lemma lower bound. \medskip

First note that by definition of the envelope $V_+$, for each state $x$ and each $\delta > 0$, one can find an open neighbourhood $U_{x,\delta}$ such that
$$
 \limsup_{\eps \to 0} \sup_{U_{x,\delta}} V_\eps \leq V_+(x) + \delta / 2.
$$

Let $\delta > 0$ be given, arbitrarily small. By definition of the infimum, there is a $x_\delta$ such that
$$
 I(x_\delta)+V_+(x_\delta) \leq \inf\b{I+V_+} + \delta/2. 
$$

We can then consider the lower bound
\begin{align*}
 \E\b{\e^{- \frac1\eps V_\eps}(\pa{X}^\eps)} & \geq \E\b{\one_{ U_{x_\delta,\delta}} \p{\pa{X}^\eps }\, \e^{- \frac1\eps V_\eps}(\pa{X}^\eps)} \\
 & \geq \P\b{\pa{X}^\eps \in U_{x_\delta,\delta}}\e^{- V_+(x_\delta)/ \eps - \delta / 2 \eps }, 
\end{align*}
and using the LDP lower bound
\begin{align*}
- \liminf_{\eps} \eps \ln \E\b{\e^{- \frac1\eps V_\eps}(\pa{X}^\eps)} & \leq \inf_{U_{x_\delta,\delta}} I + V_+(x_\delta) + \delta /2 \\
&  \leq I(x_\delta) + V_+(x_\delta) + \delta /2 \\
& \leq \inf\b{I+V_+} + \delta .
\end{align*}

\end{proof}

\begin{Lem}
\label{lem:var:upper}
 Assume $\p{\pa{X}^\eps}_{\eps > 0}$ satisfies a LDP with good rate function $I$ on a Polish state space. Let $\p{V_\eps}_{\eps > 0}$ be a family of measurable functions in $[0,+\infty]$. Define for each state $x$ the lower semi-continuous envelope of the latter by:
 \begin{equation}\label{eq:low_V}
 V_-(x) \eqdef \liminf_{\substack{\eps \to 0 \\ y \to x}} V_\eps(y) .
 \end{equation}
 Then
 $$
 -\limsup_{\eps \to 0} \eps \ln \E\b{ \e^{- \frac1\eps V_\eps}(\pa{X}^\eps)} \geq \inf\p{ I + V_- } .
 $$
\end{Lem}
\begin{proof}

This is the classical proof of Varadhan upperbound based on the rate function goodness. \medskip

Let $\delta > 0$ be given, arbitrarily small. For each $x$, by lower semi-continuity of $I$ and definition of $V_-$, we can find an open neighborhood $U_{x,\delta}$ such that:
 $$
 \inf_{\overline{U_{x,\delta}}} I \geq I(x) - \delta/2 ,
 $$
 as well as
 $$
 \liminf_{\eps \to 0}\inf_{\overline{U_{x,\delta}}} V_\eps \geq V_-(x) - \delta/2 .
 $$
 
Using the compactness of the level sets of $I$, one can choose a finite covering $U_{x_i,\delta}$, $i = 1 \ldots I$ of the level set $\set{x \mid  I(x) \leq v_{\rm max}}$ with $v_{\rm max}$ arbitrarily large, and denote $U \eqdef \bigcup_{i=1 \ldots I} U_{x_i,\delta}$. Consider now the main upper bound:

\begin{align*}
 \E\b{\e^{- \frac1\eps V_\eps}(\pa{X}^\eps)} & \leq \sum_{i=1}^{I} \E\b{\one_{ \pa{X}^\eps \in U_{x_i,\delta} } \, \e^{- V_\eps(\pa{X}^\eps)/ \eps} } + \P\b{\pa{X}^\eps \in U^c} \\
 & \leq  \sum_{i=1}^{I} \P\b{ \pa{X}^\eps \in \overline{U_{x_i,\delta}}  } \, \e^{- \inf_{\overline{U}_{x_i,\delta}}\p{V_\eps}/ \eps}  + \P\b{\pa{X}^\eps \in U^c} .
\end{align*}
By construction of the neighbourhood $U_{x,\delta}$ and the LDP upper bound
$$
-\limsup_\eps \eps \ln \P\b{\pa{X}^\eps \in \overline{U_{x,\delta}}  } \, \e^{- \inf_{\overline{U}_{x,\delta}}\p{V_\eps}/ \eps} \geq I(x) + V_-(x) - \delta,
$$
so that 
\begin{align*}
 -\limsup_\eps \eps \ln \E\b{\e^{- \frac1\eps V_\eps}(\pa{X}^\eps)} &\geq
 \min \big ( V_-(x_1)+I(x_1), \ldots ,V_-(x_I)+I(x_I), \underbrace{\inf_{U^c} I}_{\geq v_{\rm max}}   \big ) - \delta \\
 &\geq \min(\inf \b{V_- + I},v_{\rm max}) - \delta .
\end{align*}

%
%
%
%
%
%
%
%
\end{proof}

\subsection{A Large Deviations estimate} In order to analyze the (large sample size) variance of the AMS algorithm in the small noise regime, we will need precise estimates on the quantity $ \eps \ln \gamma^\eps_l(q^2_\eps)$ when $\eps \to 0$. The latter will yield the 'final' part $\loss_{\mrm{O}}(l)$ of the loss function. \medskip

%

The goal of the present section is to detail the proof of these results. The proof is based on the extended Varadhan lemmas of the previous section. \medskip

We start by defining the set $H_l(\pa{x})$
of excursions in $\set{\xi = z}$ before hitting the interior $\set{\xi > l}$:
\begin{Lem}\label{lem:lim_hit}Let $\pa{x} \in C([0,T],E)$ be given with $\tau_l^-(\pa{x}) < +\infty$ and $T$ arbitrarily large. Denote the closed set
 $$
 H_l(\pa{x}) \eqdef\set{\xi = l}  \cap \set{ \pa{x}_t, \,  t \in [\tau_l^-(\pa{x}), \tau_l^+(\pa{x})\wedge \tau_A^+(\pa{x}) \wedge T]}.
 $$
 Assume $\pa{x}^n \to_{n} \pa{x}$ in $C([0,+\infty[,E)$ for the uniform topology on bounded time intervals and that $H_l(\pa{x}) \neq\emptyset$. Let $t^n$, 
 $n\geq 1$, denotes a sequence such that for all $n\geq 1$:
 $$
 t_n \in [\tau_l^-(\pa{x}^n),\tau_l^+(\pa{x}^n) \wedge \tau_A^+(\pa{x}^n) \wedge T] \,\, \text{and} \,\, \pa{x}^n_{t^n} \in \set{\xi = l}.
 $$ 
 Then up to extraction of a sub-sequence
 $$
 \lim_n \pa{x}^n_{t^n} \in H_l(\pa{x}).
 $$
 
\end{Lem}

\begin{proof}
Let us denote $(y,s)=\lim_n (\pa{x}^n_{t^n},t^n)$, which always exists up to extraction by a compacity argument. We need to prove that $s\geq \tau_l^-(\pa{x})$, $s \leq \tau_l^+(\pa{x})\wedge \tau_A^+(\pa{x})\wedge T$ and $\pa{x}_s=y$ with $\xi(y)=l$ in order to complete the proof.\medskip

Trivially, $s \leq T$. Assume $\tau_l^+(\pa{x})\wedge \tau_A^+(\pa{x}) < T$. By definition, $\pa{x}^n$ which converges to $\pa{x}$ hits $\set{\xi > l} \cup \mathring{A}$ before $\tau_l^+(\pa{x})\wedge \tau_A^+(\pa{x})+\delta$ for all $n$ large enough and $\delta$ arbitrary small; hence $s \leq \tau_z^+(\pa{x})\wedge \tau_A^+(\pa{x}) $.\medskip

Similarly, assume $s< \tau_l^-(\pa{x})$. Then we can find a small $\delta$ and an infinite number of $\pa{x}^n$ such that $\tau_l^-(\pa{x}^n)<\tau_l^-(\pa{x})-\delta$, which contradicts the uniform convergence. \medskip

Finally, $\pa{x}^n_{t_n}$ converges towards $\pa{x}_{s}$ by uniform convergence, and since $\xi(\pa{x}^n_{t^n})=l$, $y$ also belongs to $\set{ \xi = l}$ by continuity of $\xi$.
\end{proof}

%
%
%
%

One can then consider the (measurable) potential function defined on continuous trajectories restricted to $[0,T]$:

\begin{equation}
 V^T_{l,\eps}(\pa{x}) \eqdef - \eps \ln \b{ \one_{\tau_{l}(\pa{x}) < T \wedge \tau_A(\pa{x})} \,  q^\eps \p{ \pa{x}_{\tau_{l}(\pa{x})} } },
\end{equation}
with the convention $\ln 0 = - \infty$. We will denote by $V_{l,\eps}(\pa{x}) \eqdef \lim_{T \to +\infty} V^{T}_{l,\eps}(\pa{x})$ the trivial extension to the time interval $\R_+$. \medskip

Before applying the extended Varadhan's lemmas, we need to estimate the lower and upper semi-continuous envelopes of $V_{l,\eps}$. For this purpose, we denote by 
$$
\pa{x}^{T \wedge  \tau_A^+}
$$
the restriction of $\pa{x}$ to the time interval $[0,T \wedge  \tau_A^+(\pa{x})]$ and define
\[
 \ul{V}^{T}_{l}(\pa{x}) \eqdef 
\begin{cases}
    & \inf_{ x \in H_l\big( \pa{x}^{T \wedge  \tau_A^+} \big) } U(x,\set{\xi = \lmax}), \qquad \text{if } \tau_l^-(\pa{x}) \leq T \wedge  \tau_A^+(\pa{x}), \\ 
    & + \infty, \qquad \text{else},
\end{cases}
\]
as well as 
\[
 \ol{V}^T_{l}(\pa{x}) \eqdef 
\begin{cases}
    & \sup_{ x \in H_l(\pa{x})} U(x,\set{\xi = \lmax}), \qquad \text{if } \tau_l^+(\pa{x}) < T \wedge \tau_A^-(\pa{x}) , \\ 
    & + \infty, \qquad \text{else},
\end{cases}
\]

and denote by $V_{l,\pm}(\pa{x}) \eqdef \lim_{T \to + \infty }V^{T}_{l,\pm}(\pa{x})$ their natural extensions on $\R_+$. 
%
%
%
%

\begin{Lem} 
\label{lem:V}
Let $\pa{x}$ denotes a path in $C([0,T],E)$. Under~\Cref{ass:ldp,,ass:cont,,ass:time,,ass:boundA}, one has:
 \begin{align*}
  &\limsup_{\substack{\pa{x}^n \to \pa{x} \\ \eps \to 0}} V^T_{l,\eps}(\pa{x}^n) \leq \ol{V}^T_{l}(\pa{x}). 
 \end{align*} 
as well as
 \begin{align*}
  &\liminf_{\substack{\pa{x}^n \to \pa{x} \\ \eps \to 0}} V^T_{l,\eps}(\pa{x}^n) \geq \ul{V}^{T}_{l}(\pa{x}) . 
 \end{align*} 
\end{Lem}

\begin{proof} 
Let $\pa{x}^n \to_n \pa{x}$ be a uniformly convergent sequence on $[0,T]$.
\paragraph{Upper bound}
The condition $\tau_l^+(\pa{x}) < \tau_A^-(\pa{x}) \wedge T $ defines an open subset of $C([0,T],E)$ according to Lemma~\ref{lem:lsc}, thus there is a $n_0$ above which the sequence satisfies $\tau_l(\pa{x}^n) < T \wedge \tau_A^-(\pa{x}) $. We can then use Lemma~\ref{lem:lim_hit}, together with Lemma~\ref{lem:unif_proba} giving the uniform limit of $q^\eps$, to obtain the claimed upper bound.
\paragraph{Lower bound}The condition $\tau_l^-(\pa{x}) > T \wedge \tau_A^+(\pa{x})$, defines again an open set according to Lemma~\ref{lem:lsc}, thus there is a $n_0$ above which the sequence satisfies $\tau_l(\pa{x}^n) > T$ and thus $\ul{V}^{T}_{l}(\pa{x}) = +\infty$. Otherwise, if $\pa{x}$ is outside that open set,  we use again Lemma~\ref{lem:lim_hit} , together with Lemma~\ref{lem:unif_proba} giving the uniform limit of $q^\eps$, to obtain the claimed lower bound.
\end{proof}

We can now proceed and estimate  $\limsup / \liminf_{\eps} \eps \ln  \gamma^\eps_l(q^2_\eps)$ using the extended Varadhan lemmas.

\begin{Lem}\label{lem:up} Let~\Cref{ass:ldp,,ass:cont,,ass:time,,ass:boundA} hold true. Then one has
 \begin{align*}
  &- \limsup_{\eps \to 0} \eps \ln \gamma^\eps_l(q^2_\eps) \\
  & \quad \geq \inf_{\pa{x}:  \, \tau^-_{\lmax}(\pa{x}) < \tau_A^+(\pa{x})} \, I_{[0,\tau_l^+(\pa{x})]}[\pa{x}] + 2 I_{[\tau_l^+(\pa{x}),\tau^-_{\lmax}(\pa{x})]}(\pa{x}) \p{\eqdef \ol{w}(l) }  . 
 \end{align*} 
\end{Lem}
\begin{proof}  Assume $T > 0$ given, arbitrary large. 
By definition of $ \gamma^\eps$, and using $ q_\eps\leq 1$, we get the upper bound
\begin{align*}
\gamma^\eps_l(q^2_\eps) & = \E\b{\one_{\tau_l(\pa{X}^\eps) < \tau_A(\pa{X}^\eps)} q_\eps^2\p{ \pa{X}^{\eps}_{\tau_l(\pa{X}^\eps)}) }} \\
 & \leq \E\b{\one_{\tau_l(\pa{X}^\eps) < T \wedge \tau_A(\pa{X}^\eps) } q_\eps^2\p{\pa{X}^{\eps}_{\tau_l(\pa{X}^\eps)}} } + \P\b{\tau_A(\pa{X}^{\eps}) \geq T } \\
 & = \E\b{\e^{-\frac2\eps V_{l,\eps}(\pa{X}^\eps)}} + \P\b{\tau_A(\pa{X}^{\eps,T}) \geq T } .
\end{align*}
Using Varadhan's upper bound (Lemma~\ref{lem:var:upper}) and Lemma~\ref{lem:V}, 
with~\Cref{ass:time}, we get
\begin{align*}
  - \limsup_{\eps \to 0} \eps \ln \gamma^\eps_l(q^2_\eps) \geq \min \b{\inf_{\pa{x} 
  } I_{[0,T]}[\pa{x}] + 2\ul{V}^{T}_{l}[\pa{x}] , C_T},
\end{align*}
where $\lim_{T \to + \infty} C_T = + \infty$. \medskip

Let $\delta > 0$ be given. Without loss of generality one can assume that the lower bound above is finite. As a consequence there exists $\pa{x}^\delta \in C([0,T],E)$ such that 
$$
\inf I_{[0,T]} + 2 \ul{V}^{T}_{l} \geq I_{[0,T]}[\pa{x}^\delta] + 2 \ul{V}^{T}_{l}[\pa{x}^\delta] - \delta.
$$
By definition of $\ul{V}_{l}^T$, there exists $$ t^\delta_\ast \in [\tau^-_l(\pa{x}^\delta) , \tau^+_l(\pa{x}^\delta) \wedge \tau_A^{+}(\pa{x}^\delta) \wedge T ] $$
such that $\pa{x}^\delta_{t^\delta_\ast} \in \set{ \xi = l}$ and
$$
\ul{V}^{T}_{l}[\pa{x}^\delta] \geq U(\pa{x}^\delta_{t^\delta_\ast},\set{\xi=\lmax}) - \delta .
$$
By additivity of the rate function, one also have $I_{[0,T]}[\pa{x}^\delta] \geq I_{[0,t^\delta_\ast]}[\pa{x}^\delta]$.

We can pick a $\pa{x}^\delta,t^\delta_\ast,T^\delta$ with $\tau^-_l(\pa{x}^\delta)\leq t^\delta_\ast < T^\delta <+\infty$ such that for any $T \geq T^\delta$ it holds
$$
\inf I_{[0,T]} + 2\ul{V}^{T}_{l} \geq I_{[0,t^\delta_\ast]}[\pa{x}^\delta] + 2U(\pa{x}^\delta_{t^\delta_\ast},\set{\xi=\lmax}) - \delta.
$$
We can now modify $\pa{x}^\delta$, and extend it after $t^\delta_\ast$ by a trajectory $\delta$-close to the optimal trajectories defining the cost $U^{\overline{A}}(\pa{x}^\delta_{t^\delta_\ast},\set{\xi=\lmax})$, that is
$$
U(\pa{x}^\delta_{t^\delta_\ast},\set{\xi=\lmax}) \geq I_{[t^\delta_\ast,\tau_{\lmax}^-(\pa{x}^\delta)]}[\pa{x}^\delta] - \delta,
$$
and such that $\tau_A^+(\pa{x}^\delta) > \tau_{\lmax}^-(\pa{x}^\delta)$. Finally, one gets

\begin{align*}
\inf I_{[0,T]} + 2\ul{V}^{T}_{l} &\geq I_{[0,t^\delta_\ast]}[\pa{x}^\delta] + 2I_{[t^\delta_\ast,\tau_{\lmax}(\pa{x}^\delta)]}[\pa{x}^\delta] - 3\delta \\
&\geq I_{[0,\tau_l^+(\pa{x}^\delta)]}[\pa{x}^\delta] + 2I_{[\tau_l^+(\pa{x}^\delta),\tau_{\lmax}(\pa{x}^\delta)]}[\pa{x}^\delta] - 3\delta \\
&\geq \inf_{\pa{x}:  \, \tau^-_{\lmax}(\pa{x}) < \tau_A^+(\pa{x})} \, I_{[0,\tau_l^+(\pa{x})]}[\pa{x}] + 2 I_{[\tau_l^+(\pa{x}),\tau^-_{\lmax}(\pa{x})]}(\pa{x}) - 3 \delta,
\end{align*}
where in the second line of the above one has used again the additivity of the rate functions. The fact that $\delta$ small and $T$ large are arbitrary yields the result.
\end{proof}

%
%

\begin{Lem}\label{lem:down}
 Let~\Cref{ass:ldp,,ass:cont,,ass:time,,ass:boundA} hold true. Then one has
 \begin{align*}
  &- \limsup_{\eps \to 0} \eps \ln \gamma^\eps_l(q^2_\eps) \\
  & \quad \leq \inf_{\pa{x}:  \, \tau_{\lmax}^-(\pa{x}) < \tau_{A}^-(\pa{x}) } \, I_{[0,\tau_l^-(\pa{x})]}[\pa{x}] + 2 I_{[\tau_l^-(\pa{x}),\tau^-_{\lmax}(\pa{x})]}(\pa{x}) \p{\eqdef \ol{w}(l) }  . \\
 \end{align*} 
\end{Lem}
\begin{proof} Let $T > 0$ be given. By definition of $\gamma^\eps$ and $q^\eps$ (and since $ q^\eps \leq 1$), we get the upper bound
\begin{align*}
\gamma^\eps_l(q^2_\eps) & = \E\b{\one_{\tau_l(\pa{X}^\eps) < \tau_A(\pa{X}^\eps)} \, q_\eps^2\p{ \pa{X}^{\eps}_{\tau_l(\pa{X}^\eps)}) }} \\
& \geq  \E\b{\one_{\tau_l(\pa{X}^\eps) < \tau_A(\pa{X}^\eps) \wedge T } \, q_\eps^2\p{ \pa{X}^{\eps}_{\tau_l(\pa{X}^\eps)}) }}  \\
&= \E\b{\e^{-\frac2\eps V_{l,\eps}(\pa{X}^{\eps,T})}} \\
\end{align*}
We can thus directly use Varadhan's lower bound (Lemma~\ref{lem:var:lower}) and Lemma~\ref{lem:V} to get
\begin{align*}
   - \limsup_{\eps \to 0} \eps \ln \gamma^\eps_l(q^2_\eps) \leq 
   \inf I_{[0,T]}  + 2\ol{V}^T_{l} .
\end{align*}
One can then restrict in the infimum above to trajectories verifying $\tau^{+}_{l}(\pa{x})=\tau^{-}_{l}(\pa{x})$ to get the upper bound:
\begin{align*}
 &\inf I_{[0,T]}  + 2\ol{V}^T_{l} \leq  \\
 & \quad \inf_{\tau^{-}_{l}(\pa{x}) =\tau^{+}_{l}(\pa{x})} I_{[0,T]}[\pa{x}] + 2 \ol{V}_{l}^T[\pa{x}] \eqdef A(l),
\end{align*}
and remark that by definition of $\ol{V}_{l}^T$, if $\tau^{-}_{l}(\pa{x}) =\tau^{+}_{l}(\pa{x}) < \tau_A^-(\pa{x})$, then 
$
 \ol{V}_{l}[\pa{x}] = U( \pa{x}_{\tau^{-}_{l}(\pa{x})}, \set{\xi \geq \lmax}).
$
Next, Assumption~\ref{ass:boundA} ensures that minimizing $I_{[\tau^{-}_{l}(\pa{x}),T]}(\pa{x})$ with the constraint that $\tau^{-}_{l}(\pa{x}) =\tau^{+}_{l}(\pa{x})$ yields $0$ so that:
\begin{align*}
 A(l) & = \inf_{\tau^{-}_{l}(\pa{x}) < \tau_A^-(\pa{x}) \wedge T}I_{[0,\tau^{-}_{l}(\pa{x})]}[\pa{x}] + 2 U( \pa{x}_{\tau^{-}_{l}(\pa{x})}, \set{\xi \geq \lmax})\\
 & = \inf_{\substack{\tau^{-}_{l}(\pa{x}) <  T \\ \tau^{-}_{\lmax}(\pa{x}) < \tau_A^-(\pa{x})}}I_{[0,\tau^{-}_{l}(\pa{x})]}[\pa{x}] + 2 I_{[\tau^{-}_{l}(\pa{x}),\tau^{-}_{\lmax}(\pa{x})]}[\pa{x}]
\end{align*}

It finally remains to remark that the horizon time  $T$ is arbitrary to conclude the proof.

\end{proof}

\subsection{Analysis of the loss functional}

\begin{Lem}\label{lem:updown} Let Assumption~\ref{ass:boundA} holds true. Let us define for $l \in [\xi(x_0),\lmax]$ and $C=A, \;\overline{A}$ or $\mathring{A}$: 
 \begin{equation*}
 \begin{cases}
   \ul{w}^C(l) &\dps \eqdef \inf_{\substack{ \pa{x}: \, \tau_{\lmax}^-(\pa{x}) < \tau_{C}(\pa{x}) \\ \pa{x}_0=x_0} } I_{[0,\tau_l^+(\pa{x})]}\b{\pa{x}} + 2 I_{[\tau_l^+(\pa{x}),\tau_{\lmax}^-(\pa{x})]}\b{\pa{x}}, \\
   \ol{w}^C(l)  &\dps \eqdef \inf_{\substack{  \pa{x}: \, \tau_{\lmax}^-(\pa{x}) < \tau_{C}(\pa{x}) \\ \pa{x}_0=x_0 }} I_{[0,\tau_l^-(\pa{x})]}\b{\pa{x}} + 2 I_{[\tau_l^-(\pa{x}),\tau_{\lmax}^-(\pa{x})]}\b{\pa{x}}.
\end{cases}
 \end{equation*}
 And we  simply denote $\ul{w}(l) = \ul{w}^A(l)$ and $\ol{w}(l) = \ol{w}^A(l)$.
 
 First, $\ol{w}^{\mathring{A}} = \ol{w}^{\overline{A}}$ and $\ul{w}^{\mathring{A}} = \ul{w}^{\overline{A}}$. Second, $\ol{w}$ and $\ul{w}$ are respectively the left- and right-continuous versions of the same decreasing function. In other words: $ \ol{w}(l) = \lim_{l^-}\ul{w}$ and $\ul{w}(l)=\lim_{l^+}\ol{w}$.
\end{Lem}

\begin{proof} First let us remark that by a direct application of Assumption~\ref{ass:boundA}, the definition of $\ul{w}$ and $\ol{w}$ are independent of the choice $C = \mathring{A}$ or $C= \overline{A}$.\medskip

Let $l' < l$ be given, and let $\pa{x}$ be a trajectory that reaches the final level $\lmax$. Since $\tau_{l'}^+ < \tau_l^-$, by additivity of the (non-negative) rate functions, one has
\[
  I_{[0,\tau_l^-(\pa{x})]}\b{\pa{x}} + 2 I_{[\tau_l^-(\pa{x}),\tau_{\lmax}(\pa{x})]}\b{\pa{x}} \leq  I_{[0,\tau_{l'}^+(\pa{x})]}\b{\pa{x}} + 2 I_{[\tau_{l'}^+(\pa{x}),\tau_{\lmax}(\pa{x})]}\b{\pa{x}} .
\]
Taking the infimum it yields $\ol{w}(l) \leq \ul{w}(l')$. 
\medskip

Using also additivity of rate functions, one can check that $\ol{w} \geq \ul{w}$ and that $\ol{w}$ and $\ul{w}$ are decreasing functions. This yields the result.
\end{proof}

\begin{Lem} Let Assumption~\ref{ass:boundA} and~\ref{ass:cont}. Let $x_0$ and $ \lmax$ be given. For $l\in [\xi(x_0),\lmax]$, let us denote by $u(l) \eqdef U(x_0,\set{\xi \geq l})$. Then it holds 
$$ 2 u(\lmax) - u(l) - \ol{w}(l) \leq \loss(l) \leq 2 u(\lmax) - u(l) - \ul{w}(l) , $$
The maximization problem 
$$
\sup_{l \in [\xi(x_0),\lmax]} \loss(l)
$$
has at least a maximizer $l_\ast \in ]\xi(x_0),\lmax[$, with 
$\loss(0)=\loss(\lmax)=0$.
\end{Lem}
\begin{proof}
First remark that by construction:
$$
\ul{w}(l) \leq \inf_{x \in \set{\xi =l}} U^{(l)}(x_0, \,\,. ) + 2 U(\,\,.,B)  \leq \ol{w}(l)
$$
which implies the first inequality by definition of the loss function. \medskip

Next it holds by definition of $u$:
\begin{align*}
 & \ol{w}(l)+u(l) \\
 & \quad = \inf_{\pa{x}: \, \tau_{\lmax}(\pa{x}) < \tau_A(\pa{x})} I_{[0,\tau_l^+(\pa{x})]}\b{\pa{x}} + 2 I_{[\tau_l^+(\pa{x}),\tau_{\lmax}(\pa{x})]}\b{\pa{x}} + \inf_{\pa{x}': \, \tau^+_{l}(\pa{x}') < \tau_A(\pa{x})} I_{[0,\tau_l^+(\pa{x}')]}\b{\pa{x}'}  \\
 & \quad \leq  \inf_{\pa{x}: \, \tau_{\lmax}(\pa{x}) < \tau_A(\pa{x})} 2 I_{[0,\tau_l^+(\pa{x})]}\b{\pa{x}} + 2 I_{[\tau_l^+(\pa{x}),\tau_{\lmax}(\pa{x})]}\b{\pa{x}} = 2 u(\lmax),
\end{align*}
which implies that $0 \leq 2 u(\lmax) - u(l) - \ol{w}(l)$. By construction one also has that $2 u(\lmax) - u(l) - \ul{w}(l) \leq 2 u(\lmax) < +\infty$. \medskip

Next, since by Assumption~\ref{ass:cont}, $u$ is an increasing continuous and $\ol{w}$ is a left-continuous decreasing function, while $\ul{o}$ is its right continuous version.\medskip

This implies that the bounded functions $u+\ol{w}$ and $u+\ul{o}$ i) attain their extrema, i) at the same values (either by left or right). An so it holds for $\loss$.\medskip

Finally, we remark that by construction $u(\xi(x_0))=0$ and $w(\xi(x_0)) = 2 u(\lmax)$, whereas $w(\lmax) = 2 u(\lmax)$, so that $\loss(0)=\loss(\lmax)=0$.
\end{proof}

%
%
%
%

\subsection{Small noise analysis of variance}

A key property enabling the analysis of the AMS large sample size variance is the following.

\begin{Lem}\label{lem:monotone}
The map
$$
l \mapsto \gamma_l(q^2_\lmax)
$$
is increasing.
\end{Lem}
\begin{proof} Let $l' \geq l$ be given and denote by $\{{\cal F}_s, s\geq 0\}$ the natural filtration of $X$. By definition of $\gamma_l$,
\begin{align*}
 \gamma_l(q^2_\lmax) & = \E\b{ \one_{\tau_l(\pa{X}) < \tau_A(\pa{X}) } \P_{\pa{X}_{\tau_l(\pa{X})}} ^2 \p{ \tau_{\lmax}(\pa{X}) < \tau_A(\pa{X}) } } \\
 & = \E\b{  \one_{\tau_l(\pa{X}) < \tau_A(\pa{X}) } \b{\E_{\pa{X}_{\tau_l(\pa{X})}}  \p{ \one_{\tau_{\lmax}(\pa{X}) < \tau_A(\pa{X}) } \mid {\cal F}_{\tau_{l}(\pa{X})\wedge \tau_{A}(\pa{X}) }       }}^2}\\
\end{align*}
so that applying Jensen's inequality to the conditional expectation above, using in addition the strong Markov property:
\begin{align*}
E_{\pa{X}_{\tau_l(\pa{X})}} & \p{ \one_{\tau_{\lmax}(\pa{X}) < \tau_A(\pa{X}) } \mid {\cal F}_{\tau_{l}(\pa{X})\wedge \tau_{A}(\pa{X})}}^2\\
&=E_{\pa{X}_{\tau_l(\pa{X})}}  \p{ E_{   \pa{X}_{\tau_{l'}(\pa{X})}}\p{\one_{\tau_{l'}(\pa{X}) < \tau_A(\pa{X}) }\one_{\tau_{\lmax}(\pa{X}) < \tau_A(\pa{X}) } \mid {\cal F}_{\tau_{l'}(\pa{X})\wedge \tau_{A}(\pa{X})}} \mid {\cal F}_{\tau_{l}(\pa{X})\wedge \tau_{A}(\pa{X})}}^2\\
&\leq E_{\pa{X}_{\tau_l(\pa{X})}}  \p{ \one_{\tau_{l'}(\pa{X}) < \tau_A(\pa{X}) }E_{   \pa{X}_{\tau_{l'}(\pa{X})}}\p{\one_{\tau_{\lmax}(\pa{X}) < \tau_A(\pa{X}) } \mid {\cal F}_{\tau_{l'}(\pa{X})\wedge \tau_{A}(\pa{X})}}^2 \mid {\cal F}_{\tau_{l}(\pa{X})\wedge \tau_{A}(\pa{X})}},
\end{align*}
which, once put in the former expression, finally gives 
$$
\gamma_l(q^2_\lmax)  \leq \gamma_{l'}(q^2_\lmax).
$$
\end{proof}

We can then compute 

\begin{Lem} Assume~Let~\Cref{ass:ldp,,ass:cont,,ass:time,,ass:boundA}. One has
\begin{align*}
 & \lim_{\eps \to 0} \eps \log \int_{\xi(x_0)}^{\lmax} \gamma^\eps_l\p{ \p{q^\eps_\lmax}^2} \d \p{-p^\eps_{l}} \\
 & \qquad = - \inf_{l \in [\xi(x_0),\lmax]} \b{U(x_0,\set{\xi = l}) + \inf_{\set{\xi = l}} \b{U^{(l)}(x_0, \, . \,) + 2 U(\,.\,,\set{\xi=\lmax})}}
\end{align*}
\end{Lem}
\begin{proof}
 Let us consider a finite discretization $\set{l_j}$ of the interval $[\xi(x_0),\lmax]$, and let us denote for simplicity $\gamma^\eps_l\p{ \p{q^\eps_\lmax}^2} \equiv \gamma^\eps_l $ throughout the present proof. By the monotony property of Lemma~\ref{lem:monotone}, it yields:
 \[
   \sum_{j} \gamma^\eps_{l_{j}} (p^\eps_{l_j}-p^\eps_{l_{j+1}}) \leq \int_{\xi(x_0)}^{\lmax} \gamma^\eps_l \d \p{-p^\eps_{l}} \leq \sum_{j} \gamma^\eps_{l_{j+1}} (p^\eps_{l_j}-p^\eps_{l_{j+1}}).
 \]
Using the usual rule for finite sum and any quantity $\ast_j^\eps$
$$\lim_{\eps \to 0} \eps \log \sum_j \ast_j^\eps = \max_j \lim_{\eps \to 0} \eps \log \ast_j^\eps, $$
one can use Lemmata~\ref{lem:cont:UL}, \ref{lem:up} and~\ref{lem:down} to obtain
\[
 \max_{j} \ul{u}_{l_j,l_{j+1}}  \leq \lim_{\eps \to 0} \eps \log \int_{\xi(x_0)}^{\lmax} \gamma^\eps_l \d \p{-p^\eps_{l}} \leq \max_{j} \ol{u}_{l_{j+1},l_{j+1}},
\]
where we have used the notation
$$
\ul{u}_{l_j,l_k} \eqdef u(l_{j}) + \ul{w}(l_k),
$$
taking a converging sequence of discretizations, $\max_j \abs{l_{j+1}-l_j} \to 0$, and recalling that $\ul{w}$ and $\ol{w}$ are the left- and -right continuous version of the same decreasing function while $u$ is continuous (see Lemma~\ref{lem:updown}), we can conclude.

\end{proof}

With the two lemmas above, we can finally conclude the proof of Theorem~\ref{th:maintheo}.

\begin{proof}[Proof of Theorem~\ref{th:maintheo}]
 Recall that $\sigma^2_{\eps,\mrm{ams}}$ denotes the large sample size variance of the estimator of $p^\eps$ in~\eqref{eq:var_in_th}.\medskip
 
 First, one has $\sigma^2_{\eps,\mrm{ams}} \leq - p_\eps^2 \ln p_\eps$ so that $\lim_\eps \eps \log \sigma^2_{\eps,\mrm{ams}} \leq - 2 U(x_0,\set{\xi=\lmax})$.\medskip

 Then a simple integration by parts then shows that:
 \begin{equation}\label{eq:var_bis}
  \sigma^2_{\eps,\mrm{ams}} = p_\eps^2 \ln p_\eps + 2 \int_{\xi(x_0)}^{\lmax} \gamma^\eps_l\p{q_\eps^2} \d \p{-p^\eps_{l}};
 \end{equation}
but since the term $p_\eps^2 \ln p_\eps$ is negative, we need to distinguish two cases. \medskip

First, if $\sup_l \loss(l) = 0$ that is if
$$ 2 U(x_0,\set{\xi=\lmax}) = \inf_{l \in [\xi(x_0),\lmax]} \b{U(x_0,\set{\xi = l}) + \inf_{\set{\xi = l}} \b{U(x_0, \, . \,) + 2 U(\,.\,,\set{\xi=\lmax})}} $$
one can simply bound $- p_\eps^2 \ln p_\eps \geq \sigma^2_{\eps,\mrm{ams}} \geq 2 \int_{\xi(x_0)}^{\lmax} \gamma^\eps_l\p{q_\eps^2} \d \p{-p^\eps_{l}}$ one obtains that the lower and upper bounds are logarithmically equivalent so that $\lim_\eps \eps \log \sigma^2_{\eps,\mrm{ams}} = - 2 U(x_0,\set{\xi=\lmax})$ .\medskip

Second, if $\sup_l \loss(l) > 0$, that is if
$$ 2 U(x_0,\set{\xi=\lmax}) > \inf_{l \in [\xi(x_0),\lmax]} \b{U(x_0,\set{\xi = l}) + \inf_{\set{\xi = l}} \b{U(x_0, \, . \,) + 2 U(\,.\,,\set{\xi=\lmax})}} $$ then the second term in~\eqref{eq:var_bis} dominates at logarithmic scales, yielding the claimed result.
\end{proof}

\subsection{AMS for two clones and at worse one iteration}\label{sec:N=2}
This section discusses the proof that for an AMS algorithm with $N=2$, one has
$$
\lim_{\eps \to 0} \eps \log \P\b{I_{\mrm{iter}}^{2,\eps} \leq 1} =
- \inf_{l \in [\xi(x_0),\lmax]} \b{ U(x_0,\set{\xi=l} +  \inf_{\set{\xi=l}} \b{  U(x_0,\, . \,) + 2 U(\, . \,,\set{\xi=\lmax}) } }.
$$
For the sake of concision, the proof is only sketched. \medskip

By exchangeability between the two clones, $
\P\b{I_{\mrm{iter}}^{2,\eps}\leq1}$ is twice the probability of the same event with the additional requirement that the killed clone is the clone with index $2$. Then remark by construction of the AMS algorithm that
$$
\P\b{I_{\mrm{iter}}^{2,\eps}=1} = 2\E \b{ 
\one_{\tau_{\mrm{Max}^\eps}(\pa{X^\eps}) < \tau_{A}(\pa{X}^\eps) } \p{ 
q^\eps_{\lmax}(\pa{X}^{\eps}_{\tau_{\mrm{Max}^\eps}(\pa{X}^\eps)}) 
}^2 
},
$$
where in the above $\pa{X}^{\eps}$ denotes the clone with index $1$ and $\mrm{Max}^\eps$ is the maximum level in $[\xi(x_0),\lmax]$ of the clone with index $2$, independent from $\mrm{Max}^\eps$. Using Lemma~\ref{lem:unif_proba}, $\mrm{Max}^\eps$ satisfies a LDP with good rate function $l \mapsto I_{\mrm{Max}}(l)=U(x_0, \set{\xi,l})$. By the tensorization principle $(\mrm{Max}^\eps,\pa{X}^{\eps})$ satisfy a LDP with a good rate function so that we can apply Varadhan lemmas similarly to the main estimates of this paper with 
$$
V_\eps(\pa{x},l) = - \ln\p{ \one_{\tau_{l}(\pa{x}) < \tau_{A}(\pa{x})} q^\eps_{\lmax}(\pa{x}_{\tau_{l}(\pa{x})})  },
$$
the only difference being the additional dependence in $l$ (we skip the time horizon cut-off $T$ for clarity), in which $l$ is replace by $\mrm{Max}^\eps$. Using similar technical arguments in the application of Varadhan's lemmas, we get that
\begin{align*}
& \eps \log \P\b{I_{\mrm{iter}}^{2,\eps}=1} = \eps \log 
2\E \b{ \e^{-2V_\eps(\pa{X}^\eps,\mrm{Max}^\eps) }
} \xrightarrow{\eps \to 0}. \\
& \qquad - \inf_{\substack{ l \in [\xi(x_0),\lmax] \\ \pa{x}_l \in \set{\xi = l}  } } \b{ I_{\mrm{Max}}(l) + U^{(l)}(x_0,\pa{x}_l) + 2 U(\pa{x}_l,\set{\xi \geq \lmax}) }
\end{align*}
which is precisely the claimed result.

%
%
%
%
%

\appendix
\section{More on Freidlin-Wentzell} 
In this Section, we recall the definitions of the rate function and the quasi-potential for a generic SDE~\eqref{eq:SDE} with possibly degenerate noise. \medskip

Define for $r \in H^1\p{[0,T],\R^m}$, where $H^1$ denotes the Sobolev space defined by $\frac{\dd}{\dd t} r \in L^2([0,T],\R^m)$, the following functional
\[
    I_{[0,T]}(\pa{x}) \eqdef \inf_{\stackrel{r \in H^1([0,T],\R^m)}{\dot{\pa{x}} = b(\pa{x} ) + \sigma(\pa{x}) \dot{r} }} \frac12 \int_0^T \abs{\dot{r}}^2 \dd t .
\]

%
The latter, when finite-valued, can be written explicitly as: 
\[ I_{[0,T]}(\pa{x}) = \frac12 \int_0^T \norm{\dot{\pa{x}}_t -b(\pa{x}_t)}_{  \p{\sigma\sigma^T}^{-1} (\pa{x}_t) }^2 \dd t , \]
where in the above $\p{ \sigma\sigma^T}^{-1}$ denotes the spectral pseudo-inverse so that:
$$
\norm{y}_{  \p{\sigma\sigma^T}^{-1}}^2 \eqdef \inf_{\lambda: y = \sigma \lambda } \abs{\lambda}^2 .
$$
The associated quasi potential is then given by:

\[
 U(x,y) = \inf_{ \pa{x}_0=x,\pa{x}_1=y} \int_{0}^1 \inf_{\lambda \in \R_\ast}\norm{\lambda \frac{\d\pa{x}_\theta}{\d \theta} - \frac{1}{\lambda}b(\pa{x}_\theta)}_{  \p{\sigma\sigma^T}^{-1} (\pa{x}_\theta) }^2  \dd \theta .
\]

\section*{Acknowledgement} This work has been partially supported by ANR SINEQ, ANR-21-CE40-0006.

\bibliographystyle{plain}
\bibliography{biblio}

\end{document}